\documentclass[12pt]{article}
\usepackage{amsmath,amssymb,amsthm,eucal}
\usepackage[hidelinks]{hyperref}

\usepackage{color}
\usepackage[normalem]{ulem}

\usepackage[shortlabels]{enumitem} 

\setlist[enumerate]{label=(\arabic*)}

\usepackage[all]{xy}

 \usepackage[alphabetic,abbrev]{amsrefs}

\usepackage{cleveref}

\usepackage{todonotes}

\expandafter\def\csname ver@etex.sty\endcsname{3000/12/31}

\usepackage{autonum}

\font\tenrm=cmr10

\font\bigss=cmssdc10 scaled 2300

\font\cmsslll=cmss10 at 14 pt

\DeclareMathOperator{\grad}{grad}
\newcommand\dom{dom}
\newcommand\Gsk{G_{{\rm SK}}}
\newcommand\Gc{G_\bC}
\newcommand\Lcal{\mathcal L}
\newcommand\Fcal{\mathcal F}
\newcommand\Hcal{\mathcal H}
\newcommand{\prim}[1]{{#1}^\prime}
\DeclareMathOperator{\con}{con}
\DeclareMathOperator{\red}{red}
\DeclareMathOperator{\Aff}{Aff}
\DeclareMathOperator{\Ad}{Ad}
\DeclareMathOperator{\Lie}{Lie}

\DeclareMathOperator\Id{id}
\DeclareMathOperator\scal{scal}
\makeatletter
\let\save@mathaccent\mathaccent
\newcommand*\if@single[3]{%
  \setbox0\hbox{${\mathaccent"0362{#1}}^H$}%
  \setbox2\hbox{${\mathaccent"0362{\kern0pt#1}}^H$}%
  \ifdim\ht0=\ht2 #3\else #2\fi
  }
\newcommand*\rel@kern[1]{\kern#1\dimexpr\macc@kerna}
\newcommand*\widebar[1]{\@ifnextchar^{{\wide@bar{#1}{0}}{}}{\wide@bar{#1}{1}}}
\newcommand*\wide@bar[2]{\if@single{#1}{\wide@bar@{#1}{#2}{1}}{\wide@bar@{#1}{#2}{2}}}
\newcommand*\wide@bar@[3]{%
  \begingroup
  \def\mathaccent##1##2{%
    \let\mathaccent\save@mathaccent
    \if#32 \let\macc@nucleus\first@char \fi
    \setbox\z@\hbox{$\macc@style{\macc@nucleus}_{}$}%
    \setbox\tw@\hbox{$\macc@style{\macc@nucleus}{}_{}$}%
    \dimen@\wd\tw@
    \advance\dimen@-\wd\z@
    \divide\dimen@ 3
    \@tempdima\wd\tw@
    \advance\@tempdima-\scriptspace
    \divide\@tempdima 10
    \advance\dimen@-\@tempdima
    \ifdim\dimen@>\z@ \dimen@0pt\fi
    \rel@kern{0.6}\kern-\dimen@
    \if#31
      \overline{\rel@kern{-0.6}\kern\dimen@\macc@nucleus\rel@kern{0.4}\kern\dimen@}%
      \advance\dimen@0.4\dimexpr\macc@kerna
      \let\final@kern#2%
      \ifdim\dimen@<\z@ \let\final@kern1\fi
      \if\final@kern1 \kern-\dimen@\fi
    \else
      \overline{\rel@kern{-0.6}\kern\dimen@#1}%
    \fi
  }%
  \macc@depth\@ne
  \let\math@bgroup\@empty \let\math@egroup\macc@set@skewchar
  \mathsurround\z@ \frozen@everymath{\mathgroup\macc@group\relax}%
  \macc@set@skewchar\relax
  \let\mathaccentV\macc@nested@a
  \if#31
    \macc@nested@a\relax111{#1}%
  \else
    \def\gobble@till@marker##1\endmarker{}%
    \futurelet\first@char\gobble@till@marker#1\endmarker
    \ifcat\noexpand\first@char A\else
      \def\first@char{}%
    \fi
    \macc@nested@a\relax111{\first@char}%
  \fi
  \endgroup
}
\makeatother

\renewcommand{\b}{\beta}
\renewcommand{\c}{\chi}
\renewcommand{\d}{\delta}

\newcommand{\g}{\gamma}

\renewcommand{\j}{\psi}
\renewcommand{\k}{\kappa}
\renewcommand{\l}{\lambda}
\renewcommand{\o}{\omega}
\renewcommand{\r}{\rho}

\renewcommand{\t}{\tau}
\renewcommand{\u}{\upsilon}

\renewcommand{\L}{\Lambda}
\renewcommand{\O}{\Omega}

\newcommand{\bC}{\mathbb{C}}
\newcommand{\bR}{\mathbb{R}}

\newcommand\Sp{\mathrm{Sp}}
\newcommand\GL{\mathrm{GL}}

\newcommand{\id}   {{\mathbf{1}}}
\renewcommand{\Im}{\mathop{\mathrm{Im}}}
\renewcommand{\Re}{\mathop{\mathrm{Re}}}

\newcommand\inv[1]{{#1}^{-1}}

\newcommand*\diff{\mathop{}\!{d}} 

\newcommand{\p}{\partial}

\renewcommand{\square}{\kern1pt\vbox
               {\hrule height 0.6pt\hbox{\vrule width 0.6pt\hskip 3pt
    \vbox{\vskip 6pt}\hskip 3pt\vrule width 0.6pt}\hrule height0.6pt}
    \kern1pt}

\newcount\colveccount
\newcommand*\colvec[1]{
  \global\colveccount#1
  \begin{pmatrix}
  \colvecnext
}
\def\colvecnext#1{
  #1
  \global\advance\colveccount-1
  \ifnum\colveccount>0
  \\
  \expandafter\colvecnext
  \else
  \end{pmatrix}
  \fi
}
\newtoks\rowvectoks
\newcommand{\rowvec}[2]{%
  \rowvectoks={#2}\count255=#1\relax
  \advance\count255 by -1
  \rowvecnexta}
\newcommand{\rowvecnexta}{%
  \ifnum\count255>0
    \expandafter\rowvecnextb
  \else
    \begin{pmatrix}\the\rowvectoks\end{pmatrix}
  \fi}
\newcommand\rowvecnextb[1]{%
    \rowvectoks=\expandafter{\the\rowvectoks&#1}%
    \advance\count255 by -1
    \rowvecnexta}

\DeclareMathOperator{\Span}{span}

\newcommand{\ra}{\rightarrow}

\DeclareMathOperator\Heis{Heis}

\DeclareMathOperator\Hol{Hol\;}

\newtheorem{Pb}{Problem}

\newtheorem{Th}{Theorem}[section]
\newtheorem{Ex}[Th]{Example}
\newtheorem{Prop}[Th]{Proposition}

\newtheorem{Cor}[Th]{Corollary}
\newtheorem{Lem}[Th]{Lemma}
\newtheorem{Def}[Th]{Definition}

\theoremstyle{definition}
\newtheorem{rem}[Th]{Remark}
\newcommand{\bP}{\begin{Pb}\ \ }
\newcommand{\eP}{\end{Pb}}
\newcommand{\bt}{\begin{Th}\ \ }
\newcommand{\et}{\end{Th}}
\newcommand{\bp}{\begin{Prop}\ \ }
\newcommand{\ep}{\end{Prop}}
\newcommand{\bc}{\begin{Cor}\ \ }
\newcommand{\ec}{\end{Cor}}
\newcommand{\bl}{\begin{Lem}\ \ }
\newcommand{\el}{\end{Lem}}
\newcommand{\bd}{\begin{Def}\ \ }
\newcommand{\ed}{\end{Def}}
\newcommand{\pf}{\begin{proof}[{\it Proof:\ \ }]}
\newcommand{\epf}{\end{proof}}

\newcommand{\be}{\begin{equation}}
\newcommand{\ee}{\end{equation}}

\newcommand{\arr}{\begin{array}{rlll}}
\newcommand{\ea}{\end{array}}
\newcommand{\bea}{\begin{eqnarray}}
\newcommand{\eea}{\end{eqnarray}}
\newcommand{\bean}{\begin{eqnarray*}}
\newcommand{\eean}{\end{eqnarray*}}
\renewcommand{\theequation}{\thesection.\arabic{equation}}
\catcode`@=11
\@addtoreset{equation}{section}
\catcode`@=12

\begin{document}
\rightline{LTH1121}
\rightline{}
\vskip 1.5 true cm
\begin{center}
{\bigss
ASK/PSK-correspondence and the r-map}
\vskip 1.0 true cm
{\cmsslll  V.\ Cort\'es$^1$, P.-S.\ Dieterich$^1$ and T.\ Mohaupt$^2$
} \\[3pt]
{\tenrm   $^1$Department of Mathematics\\
and Center for Mathematical Physics\\
University of Hamburg\\
Bundesstra{\ss}e 55,
D-20146 Hamburg, Germany\\
vicente.cortes@uni-hamburg.de, peter-simon.dieterich@uni-hamburg.de
\\[1em]
$^2$Department of Mathematical Sciences\\
University of Liverpool\\
Peach Street \\
Liverpool L69 7ZL, UK\\
Thomas.Mohaupt@liv.ac.uk}
\\[1em]
\today
\end{center}
\vskip 1.0 true cm
\baselineskip=18pt
\begin{abstract}
  \noindent
  We formulate a correspondence between affine and projective special K\"ahler manifolds of the same dimension. As an application, we show that, under this correspondence, the affine special K\"ahler manifolds in the image of the rigid r-map are mapped to one-parameter deformations of projective special K\"ahler manifolds in the image of the supergravity r-map. The above one-parameter deformations are interpreted as perturbative $\alpha'$-corrections in heterotic and type-II string compactifications with $N=2$ supersymmetry. Also affine special K\"ahler manifolds with quadratic prepotential are mapped to one-parameter families of projective special K\"ahler manifolds with quadratic prepotential. We show that the completeness of the deformed supergravity r-map metric depends solely on the (well-understood) completeness of the undeformed metric and the sign of the deformation parameter.
\end{abstract}

\tableofcontents

\section*{Introduction}
The supergravity c-map, described in \cite{ferrara90},  can be understood as a special case of a more general construction, the HK/QK-correspondence. In fact, the supergravity c-map can be reduced to the much simpler rigid c-map. The corresponding
manifolds and maps are summarized in the following
diagram:
\begin{equation}
\label{Diagram:c-map}
\xymatrix{
M \ar@{|->}[r]^{c} \ar@{|->}[d]_{SC}
& N  \ar@{|->}[rrr]^{con} \ar@{|->}[drrr]^{HK/QK}
& & &  \hat{N}  \ar@{|->}[d]^{SC}\\
\bar{M}  \ar@{|->}[rrrr]_{\bar{c}} & & & & \bar{N} \\
}
\end{equation}

In this diagram the scalar manifolds $\bar{M}$ of four-dimensional
vector multiplets coupled to supergravity, which
are projective special K\"ahler, are related to the
scalar manifolds
$\bar{N}$ of three-dimensional hypermultiplets coupled to
supergravity, which are quaternionic-K\"ahler,
by the supergravity $c$-map, which is induced
by dimensional reduction from four to three dimensions.
In the superconformal formulation of supergravity, the scalar manifolds $\bar M$ and $\bar N$ are obtained as superconformal quotients, denoted by $SC$ in the diagram, from the scalar manifolds $M$ and $\hat N$ of associated rigid superconformal theories.
From this viewpoint reducing the supergravity c-map to the rigid c-map requires to associate to hyper-K\"ahler manifolds $N$ in the image of the rigid c-map a hyper-K\"ahler cone $\hat N$. This operation is denoted in the diagram by $con$ and is known as conification \cite{acm12}.
The resulting relation between hyper-K\"ahler and quaternionic K\"ahler manifolds of the same dimension in the image of the rigid and local c-maps, respectively, is obtained from the HK/QK-correspondence \cites{haydys2008hyperkahler,acm12,acdm13}.
It turns out that to apply the HK/QK-correspondence it is not essential that the hyper-K\"ahler manifold is in the image of the rigid c-map but what is required is essentially a function generating a certain isometric Hamiltonian flow.
As a result one obtains not only the supergravity c-map metric but a one-parameter deformation thereof.

When attempting to apply this approach to the supergravity r-map
introduced in
\cite{dewit92}, which is induced by the dimensional reduction of five-dimensional vector multiplets to four dimensions, one runs into the following problem. Although there exists a conification procedure for K\"ahler manifolds carrying an isometric Hamiltonian flow, which could potentially be applied to our problem, it turns out that the manifolds in the image of the rigid r-map do not carry a distinguished isometric Hamiltonian flow. Even worse, applying the K\"ahler conification to any of the generically existing Hamiltonian flows does not yield the desired metric.

In this paper, we will solve this puzzle by establishing an ASK/PSK-correspondence, see \Cref{thm:conditionsglobalconification} and \Cref{def:askpsk}, relating affine special K\"ahler to projective special K\"ahler manifolds of the same dimension. This is achieved by a new conification procedure which maps affine special K\"ahler manifolds to conical affine special K\"ahler manifolds and does not require a Hamiltonian flow.
The relations between the rigid and local $r$-maps,
superconformal quotients, conification, and the ASK/PSK
correspondence are summarized in the following diagram.
\begin{equation}
\label{Diagram:r-map}
\xymatrix{
{U} \ar@{|->}[r]^{r} \ar@{|->}[d]_{SC}
& M  \ar@{|->}[rrr]^{con} \ar@{|->}[drrr]^{ASK/PSK}
& & & \hat{M}  \ar@{|->}[d]^{SC}\\
{\cal H}  \ar@{|->}[rrrr]_{\bar{r}} & & &  &\bar{M} \\
}
\end{equation}
Superconformal quotients map
conical affine special real manifolds ${U}$ to
projective special real manifolds ${\cal H}$,
and conical affine special K\"ahler manifolds $\hat{M}$
to projective special K\"ahler manifolds $\bar{M}$.
While ${U}$ and $\hat{M}$ are the scalar target
manifolds of five- and four-dimensional superconformal
vector multiplets, ${\cal H}$ and $\bar{M}$ are the
target manifolds of the gauge equivalent theories of
five- and four-dimensional
vector multiplets coupled to (Poincar\'e) supergravity.
The lift of the supergravity $r$-map $\bar r$ to the scalar manifolds of the associated superconformal vector multiplets
is the composition $con\circ r$ of the rigid $r$-map $r$
with the new conification map $con$, which will be defined and
analyzed in detail in this paper. In short, by applying
the rigid $r$-map to a conical affine special real manifold
${U}$ one obtains a K\"ahler manifold $M$ which is
affine special, but not conical. To relate $M$ to the
projective special K\"ahler manifold $\bar{M}$ obtained
by the supergravity $r$-map, we will construct a conical
affine special K\"ahler manifold $\hat{M}$ of dimension
$\dim_{\mathbb{R}} \hat{M} = \dim_{\mathbb{R}} M + 2 =
\dim_{\mathbb{R}}\bar{M}+2$ using the conification map
$con$. This provides us with a `superconformal lift'
$U \mapsto \hat{M}$ of the supergravity $r$-map and
with a correspondence $M \mapsto \bar{M}$ between
affine and projective special K\"ahler manifolds
of the same dimension, which are in the image of the
respective $r$-map. This is a special case of the
ASK/PSK correspondence.

Now we explain the geometric idea underlying the ASK/PSK-correspondence.
The initial affine special K\"ahler manifold of complex dimension $n$ can be locally realized as a Lagrangian submanifold of $\bC^{2n}$ with induced geometric data, whereas a projective special K\"ahler manifold of complex dimension $n$ is locally realized as projectivization of a Lagrangian cone in $\bC^{2n+2}$, see \cite{acd02} for these statements.
So basically we have to map a Lagrangian submanifold $\Lcal\subset\bC^{2n}$ to a Lagrangian cone in $\hat\Lcal\subset\bC^{2n+2}$. This is done in two steps.
First, we embed $\Lcal$ into the affine hyperplane $\{z^0=1\}\subset\bC^{2n+1}=\bC\times\bC^{2n}$, where $z^0$ denotes the coordinate on the first factor. Then we take $\hat\Lcal\subset\bC^{2n+2}$ to be the cone over the graph of certain function
\begin{equation}
  \bC^{2n+1} \supset \{1\}\times\Lcal\cong\Lcal\xrightarrow{f}\bC.
\end{equation}
The function $f$ is what we call a Lagrangian potential, see \Cref{def:lagrangianpotentialpair}, and is unique up to an additive constant $C$.
This constant plays a role analogous to the freedom in the choice of the Hamiltonian function in the HK/QK-correspondence \cites{acm12,acdm13}.
Whereas the real part of $C$ has no effect on the resulting geometry, changing the imaginary part gives rise to a family of projective special K\"ahler manifolds $(\widebar M_c, \widebar g_c)$ depending on the real parameter $c=\Im(C)$.
We discuss some global aspects of this construction in terms of a flat principal bundle with structure group $\Gsk=\Sp(\bR^{2n})\ltimes\Heis_{2n+1}(\bC)$. This group acts on the set of pairs $(\Lcal, f)$, where $\Lcal\subset\bC^{2n}$ is a Lagrangian submanifold and $f$ is a Lagrangian potential, and acts simply transitively on the set of special K\"ahler pairs $(\phi, F)$ consisting of a (pseudo-)K\"ahlerian Lagrangian immersion $\phi: M \to \bC^{2n}$ and a corresponding holomorphic prepotential $F$, see \Cref{def:prepotential}. For the close relation between Lagrangian potentials and holomorphic prepotentials, see \Cref{lem:lagpotprepotcorrespondence}. Note that the group $\Gsk$ is a central extension of the affine group $\Aff_{\Sp(\bR^{2n})}(\bC^{2n}) = \Sp(\bR^{2n})\ltimes\bC^{2n}$. The latter group acts simply transitively on K\"ahlerian Lagrangian immersions, and the central extension is necessary to extend this action to the holomorphic prepotentials. It turns out that the group $\Gsk$, contrary to the group $G = \Sp(\bR^{2n})\ltimes\Heis_{2n+1}(\bR)$, is not compatible with the induced K\"ahler metrics on the Lagrangian cones. It includes transformations which change the holomorphic prepotential $F$ by terms of the form $\sqrt{-1}(a_k Z^k+c)$, where $a_k$ and $c$ are real, which are not compatible with the induced metrics.

Our main application of the ASK/PSK-correspondence is a one-parameter deformation of the supergravity $r$-map obtained by applying the ASK/PSK-correspondence to an affine special K\"ahler manifold which is obtained from a conical affine special real manifold $U\subset\bR^{n}$ via the rigid $r$-map, see \Cref{thm:rmap}.
We give a global description of the resulting projective special K\"ahler manifolds $(\widebar M_c, \widebar g_c)$, where $(\widebar M_0, \widebar g_0) = (\widebar M, \widebar g)$ is the manifold in the image of the supergravity $r$-map.
The manifold $\widebar M_c$ is a domain in $\bC^n$ of the form $\widebar M_c = \bR^n+iU_c$, where $U_c \subset U$.
We analyze when $(\widebar M_c, \widebar g_c)$ is a complete Riemannian manifold.
First of all, the undeformed Riemannian manifold $(\widebar M, \widebar g)$ is complete if and only if the underlying projective special real manifold $\Hcal\subset\bR^n$ is a connected component of a global level set $\{x\in\bR^n\mid h(x)=1\}$ of a homogeneous cubic polynomial $h$ \cites{cortes2012completeness,cns16}.
Recall that the level set is required to be locally strictly convex for $\Hcal$ to be a projective special real manifold (with positive definite metric).
Assuming the undeformed metric to be complete we prove that the deformed manifold $(\widebar M_c, \widebar g_c)$, $c\neq 0$, is Riemannian and complete if and only if $c$ is negative, see \Cref{thm:rmap}.
These results should be contrasted with the more involved completeness theorems for one-loop deformed $c$-map spaces \cite{cds16}.
In the case of projective special K\"ahler manifolds with cubic prepotential the completeness of the supergravity \linebreak[3]$c$-map metric was shown to be preserved precisely under one-loop deformations with positive deformation parameter.
In case of general $c$-map spaces, however, this result has been established only under the additional assumption of regular boundary behavior for the initial projective special K\"ahler manifold, which is satisfied, for instance, for quadratic prepotentials.
As in the case of the one-loop deformed $c$-map the isometry type of the deformed r-map space $(\widebar M_c,\widebar g_c)$ depends only on the sign of $c$ (positive, negative or zero).
Note that the completeness of $\widebar M_0$ implies that $\widebar M_1$ is neither isometric to $\widebar M_0$ nor to $\widebar M_{-1}$, since the latter $2$ manifolds are then complete whereas $\widebar M_1$ is incomplete.
Computing the scalar curvature in examples, see Examples \ref{ex:ex1} and \ref{ex:ex2}, we complete this analysis by showing that $\widebar M_0$ and $\widebar M_{-1}$ are in general not isometric.
Incidentally, most, but not all, of the above results extend from cubic polynomials to general homogeneous functions, say of degree $k>1$, see \Cref{rem:eldef}.
For instance, it is not known whether the above necessary and sufficient completeness criterion for projective special real manifolds \cite{cns16}*{Theorem 2.5} holds for polynomials of quartic and higher degree.

Let us now explain how our deformation of the supergravity $r$-map can be interpreted physically as a `stringy deformation.'
Five-dimensional supergravity coupled to $n_V = n-1$ vector multiplets
(and as well hypermultiplets, which are not relevant for our
discussion) can be obtained by compactification of the
heterotic string on $K3 \times S^1$, together with a choice
of an $E_8\times E_8$ or $SO(32)$ vector
bundle $V$ \cites{Antoniadis:1995vz,Aspinwall:1996mn,Louis:1996mt},
referred to as the gauge bundle,
or by
compactification of eleven-dimensional supergravity on a
Calabi-Yau threefold \cite{Cadavid:1995bk}. The vector multiplet couplings
are encoded in a cubic, homogeneous polynomial (sometimes called cubic prepotential),
\[
h = -\frac{1}{6} C_{ijk}x^i x^j x^k \;,\;\;\;i,j,k=1, \ldots, n\;,
\]
which can
be identified up to a sign with the Hesse potential $-h=\frac16C_{ijk}x^ix^jx^k$ of a projective special real manifold (with positive definite metric). The coefficients $C_{ijk}$ depend on the details of the compactification.
For Calabi-Yau compactifications
they are the triple-intersection numbers of four-cycles, while
for heterotic compactifications they depend on the number
of vector multiplets and the gauge bundle.

Upon reduction on a further circle the Hesse potential determines a holomorphic
prepotential, with the real variables $x^i$ being replaced
by complex variables $Z^i$:
\begin{equation}\label{eq:prepot}
\hat F =  \frac{1}{6}C_{ijk} \frac{Z^i Z^j Z^k}{Z^0} \;.
\end{equation}
But while a five-dimensional supersymmetry requires that the Hesse potential must be a polynomial, four-dimensional supersymmetry only requires the prepotential to be holomorphic.
This allows further terms in Eq.~\ref{eq:prepot}, and it turns out that such terms
are created by $\alpha'$-corrections. The dimensional
reductions of the constructions discussed above give rise
to heterotic string theory on $K3\times T^2$ and type-IIA
string theory on a Calabi-Yau threefold. The prepotential,
including corrections takes the form
\cites{Candelas:1990rm,Hosono:1993qy,Ceresole:1995jg,deWit:1995zg,Antoniadis:1995ct,Harvey:1995fq}
\[
\hat F =  \frac{1}{6}C_{ijk} \frac{Z^i Z^j Z^k}{Z^0}
- 2\sqrt{-1}c  (Z^0)^2 + \cdots \;,
\]
where the omitted terms are exponentially small for
large $\mbox{Re}(Z^i/Z^0)$ and the factor $-2$ corresponds to the factor of $-2$ in formula \eqref{eq:Fkorrektur}. In type-II Calabi-Yau compactifications the omitted terms are world-sheet intstantons and, therefore, non-per\-tur\-bative
corrections in $\alpha'$. The leading correction term $-2\sqrt{-1}c(Z^0)^2$
arises at four-loop level in $\alpha'$ perturbation
theory \cites{Grisaru:1986kw,Nemeschansky:1986yx,Candelas:1990rm},
and the real coefficient $c$ is proportional to $\zeta(3) \chi$,
where $\zeta$ is the Riemann $\zeta$-function and $\chi$ is
the Euler number of the Calabi-Yau three-fold.
The heterotic prepotential has an analogous structure, and the
coefficient $c$
is proportional to $\zeta(3) c_1(0)$, where
$c_1(0)$, as well as the coefficients of the further correction
terms, is obtained by expanding a
(model-dependent) modular form \cite{Harvey:1995fq}.

We have mentioned that when performing the conification we can shift
the Lagrangian potential (or, equivalently, the holomorphic prepotential $F=\frac16C_{ijk}z^i z^j z^k$ of the initial affine special K\"ahler manifold) by an imaginary constant, which then deforms
the resulting prepotential by precisely the same type of term
as is created by the leading $\alpha'$-correction.
Thus the
resulting deformed supergravity r-map might be called a
`stringy' r-map.
We remark that the further freedom to also include imaginary translations
does not have an interpretation in the above string theory
realizations. Imaginary translations correspond to adding terms
\[
\Delta \hat F = i a_{0I} Z^0 Z^I
\]
to the prepotential, where $a_{0I}$ are real constants. Such terms do not occur as quantum or
stringy corrections in the above four-dimensional string models.
Curiously, adding a term
\[
  \delta \hat
  F = \frac{1}{24}c_{2I} Z^0 Z^I
\]
to the type-IIA prepotential,
where $c_{2I}$ are the components of the second Chern class of $X$,
has been discussed before in the literature.
However, this term has a real coefficient and can be
transformed away by a symplectic transformation. Conversely, it
can be generated by a symplectic transformation, which
was used in \cite{Behrndt:1996jn} as a solution-generating
technique for black hole solutions.

\section{Preliminaries}
\begin{Def}
\label{def:ask}
  An \emph{affine special K\"ahler manifold} $(M,J,g,\nabla)$ is a pseudo-K\"ahler manifold $(M,J,g)$ with symplectic form $\omega := g(\cdot,J\cdot)$ endowed with a flat torsion-free connection $\nabla$ such that $\nabla\omega = 0$ and $\diff^\nabla J = 0$.
\end{Def}


\begin{Def}
\label{def:kaehlerianlagrangiantotallycomplex}
  Let $M$ be a complex manifold of complex dimension $n$ and consider the complex vector space $T^*\bC^{n} = \bC^{2n}$ endowed with the canonical coordinates $(z^1,\ldots,z^n,\allowbreak   w_1,\ldots, w_n)$, standard complex symplectic form $\Omega = \sum_{i=1}^n \diff z^i\wedge\diff w_i$, standard real structure $\tau: \bC^{2n}\to \bC^{2n}$ and Hermitian form $\gamma = \sqrt{-1}\Omega(\cdot,\tau\cdot)$. A holomorphic immersion $\phi:M\to \bC^{2n}$ is called \emph{Lagrangian} (respectively, \emph{K\"ahlerian}) if $\phi^*\Omega = 0$ (respectively, if $\phi^*\gamma$ is non-degenerate). $\phi$ is called \emph{totally complex} if $\diff\phi(T_pM)\cap \tau\diff\phi(T_pM) = 0$ for all $p\in M$.
\end{Def}

\begin{Prop}[\cite{acd02}]
  \label{prop:kaehlerianifftotallycomplex}
  Let $\phi: M \to \bC^{2n}$ be a holomorphic immersion.
  \begin{enumerate}
  \item $\phi$ is totally complex if and only if its real part $\Re\phi: M \to \bR^{2n}$ is an immersion.
  \item If $\phi$ is Lagrangian, then $\phi$ is K\"ahlerian if and only if it is totally complex.
\end{enumerate}
\end{Prop}

A K\"ahlerian Lagrangian immersion $\phi: M \to \bC^{2n}$ induces on $M$ the structure of an affine special K\"ahler manifold. Locally, an affine special K\"ahler manifold can always be realized as a K\"ahlerian Lagrangian immersion. This is reflected in the following proposition.

\begin{Prop}[\cite{acd02}]
\label{prop:askimmersion}
Let $(M,J,g,\nabla)$ be a simply connected affine special K\"ahler manifold of complex dimension $n$. Then there exists a K\"ahlerian Lagrangian immersion $\phi: M \to \bC^{2n}$ inducing the affine special K\"ahler structure $(J,g,\nabla)$ on $M$. Moreover, $\phi$ is unique up to a transformation of $\bC^{2n}$ by an element in $\Aff_{\Sp(\bR^{2n})}(\bC^{2n})$.
\end{Prop}

More precisely, the action of the group $\Aff_{\Sp(\bR^{2n})}(\bC^{2n})$ on the set of K\"ahlerian Lagrangian immersions $\phi: M \to\bC^{2n}$ is simply transitive, as can be proven along the lines of the proof of simple transitivity in \Cref{prop:gactiononprepotentials}.




\begin{Def}
\label{def:prepotential}
  Let $\phi: M \to \bC^{2n}$ be a K\"ahlerian Lagrangian immersion of an affine special K\"ahler manifold $M$. Denote by $\lambda = w^t\diff z = \sum_{i=1}^n w_i\diff z^i$ the Liouville form of $\bC^{2n}$. A function $F: M\to\bC$ is called a \emph{prepotential} of $\phi$ if $\diff F = \phi^*\lambda$.
\end{Def}
\begin{rem}
  \label{rem:ask}
  \begin{enumerate}
  \item The function $K := \frac12\gamma(\phi,\phi)$ is a K\"ahler potential of the K\"ahler form $\omega$, i.e., $\omega = -\frac i2\partial\bar\partial K$.
  \item Let $M$ be a local affine special K\"ahler manifold given as a K\"ahlerian Lagrangian immersion $\phi: M \to\bC^{2n}$. Then the pullback of the canonical coordinates of $T^*\bC^{n}=\bC^{2n}$ gives functions $z^1,\ldots,z^n,w_1,\ldots,w_n : M\to\bC$ such that $\phi=(z,w):=(z^1,\ldots,z^n,w_1,\ldots,w_n)$. It can always be achieved that $z,w:M\to\bC^n$ are holomorphic coordinate systems by replacing $\phi$ with $ x\circ\phi$ for some $x\in\Sp(\bR^{2n})$ and restricting $M$ if necessary \cite{acd02}. In this case, we call $(z,w)$ a \emph{conjugate pair of special holomorphic coordinates}.
  \item Let $\phi=(z,w):M\to\bC^{2n}$ be a K\"ahlerian Lagrangian immersion of an affine special K\"ahler manifold given by a conjugate pair of special holomorphic coordinates $(z,w)$ and let $F: M\to\bC$ be a prepotential of $\phi$. Then we can identify $M \cong z(M)\subset\bC^{n}$ and $\phi$
    with $\diff F: M\to T^*M = \bC^{2n}$. In particular, $\phi(M) = \left\{ (z,w)\in\bC^{2n} \mid w_i = \frac{\partial F}{\partial z^i} \right\}$ is the graph of $\diff F$ over $M$. In this case, $M\subset\bC^{n}$ is called an \emph{affine special K\"ahler domain} and $K(p) = \sum_{i=1}^n\Im(\widebar z^i F_i)$ where $F_i := \frac{\partial F}{\partial z^i}$.
    \label{rem:item:askdomain}
  \end{enumerate}
 \end{rem}

 \begin{Def}
   A \emph{conical affine special K\"ahler manifold} $(\hat M,\hat J,\hat g,\hat\nabla,\xi)$ is an affine special K\"ahler manifold $(\hat M,\hat J,\hat g,\hat\nabla)$ and a vector field $\xi$ such that $\hat g(\xi,\xi) \neq 0$ and $\hat \nabla\xi=\hat D\xi=\Id$, where $\hat D$ is the Levi-Civita connection of $\hat g$.
 \end{Def}

   Note that contrary to \cite{cortes2012completeness}*{Definition 3} here we are not making any assumptions on the signature of the metric $\hat g$.

 A conical affine special K\"ahler manifold $\hat M$ of complex dimension $n+1$ locally admits K\"ahlerian Lagrangian immersions $\Phi: U \to \bC^{2n+2}$ that are equivariant with respect to the local $\bC^*$-action defined by $Z = \xi - iJ\xi$ and scalar multiplication on $\bC^{2n}$ \cite{acd02}. As a consequence, the function $\hat K := \frac12\hat g(Z,\widebar Z) = \hat g(\xi, \xi)$ is a globally defined K\"ahler potential of $\hat M$. Indeed, if $p\in U$ then $2\hat K = \hat g(Z,\widebar Z) = \hat \gamma(\Phi,\Phi)$, where $\hat\gamma$ is the standard Hermitian form of $\bC^{2n+2}$.

 If the vector field $Z$ generates a principal $\bC^*$-action then the symmetric tensor field
 \begin{equation}
   \label{eq:pskmetric}
   g' := -\frac {\hat g}{\hat K} + \frac{(\partial \hat K)(\bar\partial \hat K)}{\hat K^2}
 \end{equation}
 induces a K\"ahler metric $\widebar g$ on the quotient manifold $\widebar M := \hat M/\bC^*$, compare \cite{cds16}*{Proposition 2}. It follows that $\pi^*\widebar g = g'$ and $\pi^*\widebar\omega = \frac i2\partial\bar\partial\log|\hat K|$, where $\widebar\omega = \widebar g(\cdot, J\cdot)$ is the K\"ahler form of $\widebar M$.
 Set $\mathcal D := \Span\{\xi, J\xi\}$. Note that if $\hat K > 0$, then the signature of $\widebar g$ is minus the signature of $\hat g\vert_{\mathcal D^\perp}$, whereas if $\hat K < 0$ then the signature of $\widebar g$ agrees with the signature of $\hat g\vert_{\mathcal D^\perp}$.

 \begin{Def}
   The quotient $(\widebar M, \widebar g)$ is called a \emph{projective special K\"ahler manifold}.
 \end{Def}

 \begin{rem}
   Let $\Phi = (Z,W): M\to\bC^{2n+2}$ be an equivariant K\"ahlerian Langrangian immersion such that $(Z,W)$ is a conjugate pair of special holomorphic coordinates. Identify $M\cong Z(M)\subset \bC^{n+1}$. Then the prepotential $F: M\to\bC$ can be chosen to be homogeneous of degree 2 such that $\Phi = \diff F$.
 \end{rem}

\section{Symplectic group actions}
\subsection{Linear representation of the central extension of the affine symplectic group}
\label{sec:line-repr-centr}

Let $G= \Sp (\bR^{2n}) \ltimes \Heis_{2n+1}(\bR)$ be the extension of the real  Heisenberg group by the group of automorphisms $\Sp (\bR^{2n})$. The complexification of $G$ is the group $\Gc = \Sp(\bC^{2n}) \ltimes \Heis_{2n+1}(\bC)$ which contains $G$ as a real subgroup. Given two elements $x=(X,s,v)$ and $x'=(X',s',v')\in \Gc$, where $X, X' \in \Sp (\bC^{2n})$, $s, s' \in \bC =Z(G)$, $v, v' \in \bC^{2n}$, their product in $\Gc$ is given by
\begin{equation}
  \label{eq:Gmultiplication}
  x\cdot x' = \left( XX', s+s' + \frac12 \O (v,Xv') , Xv'+v\right),
\end{equation}
where $\O$ is the symplectic form on $\bC^{2n}$.

The group $\Gc$ is a central extension of the group $\Aff_{\Sp(\bC^{2n})}(\bC^{2n})$ of affine transformations of $\bC^{2n}$ with linear part in $\Sp(\bC^{2n})$. The two groups are related by the quotient homomorphism
\begin{equation}
\Gc \ra \Aff_{\Sp (\bC^{2n})}(\bC^{2n})= \Gc/Z(\Gc), \quad (X,s,v) \mapsto (X,v).\label{eq:defrhobar}
\end{equation}
This induces an affine representation $\bar{\rho}$ of $\Gc$ on $\bC^{2n}$ with image $\Aff_{\Sp(\bC^{2n})}(\bC^{2n})$ whose restriction to the real group $G$ has the image  $\bar\rho(G) = \Aff_{\Sp(\bR^{2n})}(\bR^{2n})$. In the complex symplectic vector space $\bC^{2n}$ we use standard coordinates $(z^1,\ldots ,z^n,w_1,\ldots ,w_n)$ in which the complex symplectic form is $\O = \sum dz^i\wedge dw_i$.

We will now show that $\bar\rho$ can be extended to a linear symplectic representation
\[ \rho: \Gc \ra \Sp(\bC^{2n+2}) \]
in the sense that the group $\rho (\Gc)$ preserves the affine hyperplane $\{ z^0=1\} \subset \bC^{2n+2}$ with respect to standard coordinates $(z^0,w_0, z^1\ldots , z^n,w_1,\ldots w_n)$ on $\bC^{2n+2}=\bC^2 \oplus \bC^{2n}$ and the distribution spanned by $\p_{w_0}$ inducing on the symplectic affine space $\{ z^0=1\}/\langle \p_{w_0}\rangle \cong \bC^{2n}$ the symplectic affine representation $\bar\rho$.

\begin{rem}
\label{rem:symplecticreductionexplained}
Notice that $\{ z^0=1\}/\langle \p_{w_0}\rangle$ is precisely the symplectic reduction of $\bC^{2n+2}$ with respect to the holomorphic Hamiltonian
group action generated by the vector field $\p_{w_0}$. The group $\rho (\Gc)\subset \Sp(\bC^{2n+2})$ preserves the Hamiltonian $z^0$ of that action and, hence,
$\rho$ induces a symplectic affine representation on the reduced space.  Similarly, we will consider the initial
real symplectic affine space $\bR^{2n}$ as the symplectic reduction of the real symplectic vector space $\bR^{2n+2}$ in the context of the real group $G$.
\end{rem}

\begin{Prop}
\label{prop:liftofaffinegroup}
\begin{enumerate}
\item[(i)] The map
\begin{eqnarray*}
&&x=\left( X,s,v\right)  \mapsto \rho (x) = \begin{pmatrix} 1&0&0\\
-2s&1&\hat{v}^t\\
v&0 &X
\end{pmatrix},\quad
 \hat{v} := X^t\O_0 v = \O_0 X^{-1} v,
  \end{eqnarray*}
  where $\O_0=\begin{pmatrix} 0&\id \\
    -\id &0\end{pmatrix}$ is the matrix representing the symplectic form on $\bC^{2n}$, defines a faithful linear symplectic representation $\rho : \Gc \ra  \Sp(\bC^{2n+2})$, which induces the affine symplectic representation $\bar\rho : \Gc \ra \Aff_{{\Sp (\bC^{2n})}}(\bC^{2n})$ in the sense explained above.
\item[(ii)] The image $\rho  (\Gc) \subset  \Sp(\bC^{2n+2})$ consists of the transformations in $\Sp(\bC^{2n+2})$ which preserve the hyperplane $\{ z^0=1\} \subset \bC^{2n+2}$ and the complex rank one distribution $\langle \p_{w_0}\rangle$. The image $\rho(G)\subset\Sp(\bR^{2n+2})\subset\Sp(\bC^{2n+2})$ is the group that additionally preserves the real structure of $\bC^{2n+2}$.
\end{enumerate}
\end{Prop}
\pf
We first observe that, for $\mathbb K\in\{\bR,\bC\}$, an element of $\GL (2n+2,\mathbb K)$ preserves $\{ z^0=1\}$ and $\langle \p_{w_0} \rangle$ if and only if it is of the form
\[
  \begin{pmatrix}
    1&0&0\\
    -2s&c&w^t\\
    v&0 &X
  \end{pmatrix},
\]
where $s\in \mathbb K$, $0\neq c\in \mathbb K$, $v, w\in \mathbb K^{2n}$, and $X\in \GL (2n,\mathbb K)$. One then checks that such a  transformation is symplectic if and only if $X\in \Sp ({\mathbb K}^{2n})$, $c=1$, and $w=\hat{v}$. Clearly an element in $\GL(2n, \mathbb K)$ preserves the real structure of $\bC^{2n}$ if and only if $\mathbb K = \bR$. This proves (ii) and shows that the linear transformation $\rho (x)$ induces the affine transformation $\bar{\rho} (x)\in \Aff_{{\Sp (\mathbb \bC^{2n})}}(\mathbb \bC^{2n})$ for all $x\in \Gc$.

 To check that $\rho$ is a representation we put $\mu (x) := -2s$, $\g (x) := \hat{v} =  X^t\O_0 v$. Then we compute \[ \mu (xx')= \mu (x) + \mu (x') -\o (v,Xv') = \mu (x) + \mu (x') + \hat{v}^tv',\]
which coincides with the matrix element of $\rho (x) \rho (x')$ in the second row and first column.
Next we compute the column vector
\[
  \g (xx') = (XX')^t\O_0 (v+Xv')
  = (X')^t(\g (x) + \O_0 v') = (X')^t\g (x) + \g (x'),
\]
the entries of which coincide with the last $2n$ entries of the second row of $\rho (x) \rho (x')$.
From these properties one sees immediately that $\rho$ is a representation. It is obviously faithful, since $X$, $s$, and $v$ appear in the matrix $\rho (x)$.
\epf

We define the subgroup $\Gsk = \Sp(\bR^{2n}) \ltimes \Heis_{2n+1}(\bC)\subset \Gc$ to be the extension of the complex Heisenberg group by $\Sp(\bR^{2n})$. It contains the real group $G$ as a subgroup and is a central extension of the affine group $\bar\rho(\Gsk) = \Aff_{\Sp(\bR{^{2n}})}(\bC^{2n})$. We will show that $\Gsk$ acts on pairs $(\phi, F)$ of K\"ahlerian Lagrangian immersions and prepotentials. This gives a transformation formula, see Eq.~\eqref{eq:prepottransform}, of prepotentials of affine special K\"ahler manifolds which generalizes de Wit's formula (9) in \cite{dewit} from the case of linear to affine symplectic transformations.

\subsection{Representation of \texorpdfstring{$G_\bC$}{Gc}  on Lagrangian pairs}
\label{sec:repr-g-lagr}
Let $\Lcal\subset\bC^{2n}$ be a Lagrangian submanifold and denote by $\eta$ be the canonical $\Sp(\bR^{2n})$-invariant 1-form given by $\eta_q := \Omega(q,\cdot)$, for $q\in\bC^{2n}$. In Darboux coordinates $(z^1,\ldots,z^n,\allowbreak w_1,\ldots,w_n)$ we can write $\eta$ as $\eta = \sum z^i\diff w_i - w_i\diff z^i$. Since $\diff\eta = 2\Omega$, this form is closed when restricted to $\Lcal$.

\begin{Def}
\label{def:lagrangianpotentialpair}
  We call a function $f: \Lcal\to\bC$ a \emph{Lagrangian potential} of
  $\Lcal$ if $\diff f = -\eta\vert_{\Lcal}$ and a pair
  $(\mathcal L, f)$ a \emph{Lagrangian pair} if $\Lcal\subset\bC^{2n}$
  is a Lagrangian submanifold and $f$ is a Lagrangian potential of
  $\Lcal$.
\end{Def}

\begin{Prop}
  \label{prop:Gactiononpairs}
  The group $\Gc$ acts on the set of pairs $(\Lcal, f)$, where $\Lcal\subset\bC^{2n}$ is a Lagrangian submanifold and $f$ is a holomorphic function on $\Lcal$. The action is defined as follows. Given $x=(X,s,v)\in \Gc$ and a pair $(\Lcal, f)$ as above, we define
  \begin{equation}
    \label{eq:Gactiondef}
    x\cdot (\Lcal, f) := (x\Lcal, x\cdot f),
  \end{equation}
  where $x\Lcal := \bar\rho(x)\Lcal$ and $x\cdot f$ is function on $x\Lcal$ defined as
  \begin{equation}
    \label{eq:transformoflagrangianpotential}
    x\cdot f := f\circ\inv x + \Omega(\cdot, v) - 2s.
  \end{equation}
  Moreover, if $f$ is a Lagrangian potential of $\Lcal$, then $x\cdot f$ is a Lagrangian potential of $x\mathcal L$.
\end{Prop}
\begin{proof}
  For the neutral element $e\in \Gc$, clearly $e\cdot(\Lcal, f) = (\Lcal, f)$.
  Let $q\in\Lcal$ and $x, x' \in \Gc$ with $x=(X,s,v)$ and $x'=(X',s',v')$. Then
  \begin{align}
    x\cdot(x'\cdot f)(xx'q)
    &=
      (x'\cdot f)(x'q) + \Omega(xx'q,v) - 2s \\
    &=
      f(q) + \Omega(x'q,v') + \Omega(xx'q, v) - 2s - 2s' \\
    &=
      f(q) + \Omega(xx'q, v + Xv') - 2\left(s+s'+\frac12\Omega(v,Xv')\right) \\
    &=
      (xx')\cdot f(xx'q),
  \end{align}
  where we have used the second-to-last equation that
  \begin{align}
  \Omega(x'q,v')
    &=
      \Omega(Xx'q, Xv') \\
    &=
      \Omega(xx'q-v, Xv') \\
    &=
      \Omega(xx'q, Xv') - \Omega(v,Xv').
  \end{align}
  This shows that Eq.~\eqref{eq:Gactiondef} defines an action of $\Gc$. Now let $f$ be a Lagrangian potential of $\Lcal$ and set $\tilde q = xq$. Then
  \begin{align}
    \diff_{\tilde q}(x\cdot f)
    &=
      \diff_q f \circ d(\inv x) + \diff_{\tilde q}(\Omega(\cdot, v)) \\
    &=
      -\eta_q\circ\inv X + \Omega(\cdot, v) \\
    &=
      -\Omega(q, \inv X\cdot) + \Omega(\cdot, v) \\
    &=
      -\Omega(Xq + v, \cdot) = -\eta_{\tilde q},
  \end{align}
  hence, $x\cdot f$ is a Lagrangian potential of $x\cdot\Lcal$.
\end{proof}

\begin{Def}
  We call a Lagrangian submanifold $\Lcal\subset\bC^{2n}$ \emph{K\"ahlerian} if the Hermitian form $\gamma = \sqrt{-1}\Omega(\cdot, \tau\cdot)$ is non-degenerate when restricted to $\Lcal$. Similarly, a Lagrangian pair $(\Lcal, f)$ is called \emph{K\"ahlerian} if $\Lcal$ is K\"ahlerian.
\end{Def}

\begin{Lem}
  A Lagrangian submanifold $\Lcal\subset\bC^{2n}$ is K\"ahlerian if and only if $\Lcal$ is totally complex, i.e., $T_q\Lcal \cap \tau T_q\Lcal = \{0\}$ for all $q\in\Lcal$.
\end{Lem}
\begin{proof}
  Since the inclusion $\iota:\Lcal\to\bC^{2n}$ is a holomorphic Lagrangian immersion, the statement follows from Prop.~\ref{prop:kaehlerianifftotallycomplex}.
\end{proof}

\begin{Cor}
\label{cor:klpaction}
  The group $\Gsk\subset\Gc$ acts on the set of K\"ahlerian Lagrangian pairs.
\end{Cor}
\begin{proof} The group $\Gsk$ acts on $\bC^{2n}$ as the group $\bar\rho(\Gsk) = \Aff_{\Sp(\bR^{2n})}(\bC^{2n})$ which is the affine linear group that leaves invariant the complex symplectic form $\Omega$ and the real structure $\tau$ and, hence, also the Hermitian form $\gamma = \sqrt{-1}\Omega(\cdot,\tau\cdot)$. This shows that if $(\Lcal, f)$ is a K\"ahlerian Lagrangian pair, then $x\cdot(\Lcal, F) = (\bar\rho(x)\Lcal, x\cdot f)$ is again a K\"ahlerian Lagrangian pair for all $x\in\Gsk$.
\end{proof}

\subsection{Representation of \texorpdfstring{$\Gsk$}{Gsk} on special K\"ahler pairs}

\begin{Def}
  Let $(M, J, g, \nabla)$ be a connected affine special K\"ahler manifold of complex dimension $n$ and let $U\subset M$ be an open subset of $M$. We call a pair $(\phi, F)$ a \emph{special K\"ahler pair} of $U$ if $\phi: U\to\bC^{2n}$ is a K\"ahlerian Lagrangian immersion inducing on $U$ the restriction of the special K\"ahler structure $(J, g, \nabla)$ and $F$ is a prepotential of $\phi$. We denote the set of special K\"ahler pairs of $U$ by $\Fcal(U)$.
\end{Def}

The following Lemma shows how the notions of prepotentials and Lagrangian potentials are related.

\begin{Lem}\label{lem:lagpotprepotcorrespondence}
  Let $M$ be a special K\"ahler manifold together with a K\"ahlerian Lagrangian embedding $\phi: M \to \phi(M)\subset\bC^{2n}$ inducing the special K\"ahler structure of $M$. Set $\Lcal := \phi(M)$ and $(z,w) := \phi$. Then a Lagrangian potential $f$ of $\Lcal$ defines a prepotential $F$ of $\phi$ via
  \begin{equation}\label{eq:lagpotprepotcorrespondence}
    F = \frac12(\phi^* f + z^t w),
  \end{equation}
  and vice versa.
\end{Lem}
\begin{proof}
  Let $f$ be a Lagrangian potential of $\Lcal$. We compute
  \begin{align}
    \diff F
    &= \frac12(\phi^*\diff f + \diff(z^t w)) \\
    &= \frac12(-\phi^*\eta + w^t \diff z + z^t \diff w) \\
    &= \frac12(w^t\diff z - z^t\diff w + w^t\diff z + z^t \diff w) \\
    &= w^t \diff z.
  \end{align}
  Since $\phi$ is a biholomorphism onto its image, the converse follows easily.
\end{proof}

\begin{Prop}
\label{prop:gactiononprepotentials}
Let $M$ be a connected affine special K\"ahler manifold of complex dimension $n$ and $U\subset M$ an open subset such that $\Fcal(U) \neq \emptyset$. Then the group $\Gsk$ acts simply transitively on $\Fcal(U)$. The action is defined as follows. Given $x=(X,s,v)\in \Gsk$ and a special K\"ahler pair $(\phi, F)$ of $U$, we define
  \begin{equation}
    \label{eq:Gactionprepotdef}
    x\cdot(\phi, F) := (x\phi, x\cdot F),
  \end{equation}
  where $x\phi:=\bar\rho(x)\circ\phi$ and
  \begin{equation}
    \label{eq:prepottransform}
    x\cdot F := F - \frac12z^tw + \frac12z'^tw' + \frac12(x\phi)^*\Omega\left(\cdot,v\right) - s,
  \end{equation}
  where $(z,w):=\phi$ and $(z',w'):=x\phi$ are the components of $\phi$ and $x\phi$, respectively.
\end{Prop}
\begin{proof}
  We begin by showing that \cref{eq:Gactionprepotdef} defines a $\Gsk$-action on $\Fcal(U)$. Clearly, the neutral element of $\Gsk$ acts trivially.
  We can locally rewrite \cref{eq:prepottransform} as
  \begin{align}
    2x\cdot F - z'^tw'
    &= 2F - z^tw + (x\phi)^*\Omega(\cdot, v) - 2s \\
    &= (x\phi)^* (f\circ\inv x + \Omega(\cdot, v) - 2s) \\
    &= (x\phi)^* (x\cdot f)
  \end{align}
  where $f$ is the Lagrangian potential locally corresponding to $F$ according to \Cref{lem:lagpotprepotcorrespondence}, i.e., $\phi^*f = 2F - z^tw$. This shows that $x\cdot F$ is a prepotential, namely the prepotential locally corresponding to the Lagrangian potential $x\cdot f$ via $x\phi$. The remaining group action axioms now follow easily from \Cref{prop:Gactiononpairs}.

  It remains to show that the action is simply transitive. Let $(\phi, F)$, $(\phi', F')$ be two special K\"ahler pairs of $U$. Since $\phi$ and $\phi'$ are both K\"ahlerian Lagrangian immersions inducing same special K\"ahler structure, we know from Prop.~\ref{prop:askimmersion} that there is an element $(X,v)\in\Aff_{\Sp(\bR^{2n})}(\bC^{2n})$ such that $\phi' = (X,v)\circ\phi$. Since prepotentials are unique up to a constant, there is an $s\in\bC$ such that $x\cdot F = F'$ for $x=(X,s,v)\in \Gsk$. This shows $x\cdot(\phi, F) = (\phi', F')$ and, hence, the transitivity.
  To see that the action is free, assume that $x\cdot(\phi, F) = (\phi, F)$ for some $x=(X,s,v)\in \Gsk$. Then $X\circ\phi + v = \phi$. Differentiating and taking the real part gives $(X-\id_{2n})\circ\Re\diff\phi = 0$. Since $\phi$ is K\"ahlerian, $\Re\phi$ is an immersion and this implies $X=\id_{2n}$. But then from $X\circ\phi + v = \phi$ it also follows that $v=0$. Finally, $x\cdot(\phi, F) = (\phi, F - s)$ implies $s=0$ and, hence, $x$ is the identity of $\Gsk$.
 \end{proof}

\begin{Cor}
\label{cor:lineartransforms}
Under the assumptions of Prop.~\ref{prop:gactiononprepotentials}, the subgroup $\Sp(\bR^{2n})\subset \Gsk$ acts by
\[
  x\cdot(\phi, F) = \left(\phi' = x\phi, F' = x\cdot F = F-\frac12 z^t w + \frac12 z'^t w'\right)
\]
on the set of special K\"ahler pairs $(\phi,F)$. In particular, in the case of conical affine special K\"ahler manifolds, $\Sp(\bR^{2n})$ acts on the set of homogeneous prepotentials of degree 2.
\end{Cor}

\begin{rem}
  By Corollary \ref{cor:lineartransforms}, the function $F-\frac12z^tw$ is invariant under the above action of $\Sp(\bR^{2n})$ in the sense that
  \begin{equation}
    \label{eq:dewitinvariance}
    F' - \frac12 z'^t w' = F-\frac12z^t w.
  \end{equation}
  This is precisely the statement of de Wit, see eq.~(10) in~\cite{dewit}, that $F-\frac12z^t w$ transforms as a symplectic function under linear symplectic transformations.

  In terms of the Lagrangian potentials $f$ and $f'$ corresponding to $F$ and $F'$, eq. \eqref{eq:dewitinvariance} is equivalent to
  \[
    f\circ\phi = f'\circ\phi'.
  \]

\end{rem}

\section{Conification of Lagrangian submanifolds}
\label{sec:conif-lagr-subm}
The aim is to associate (under some assumptions) a Lagrangian cone $\hat \Lcal \subset \bC^{2n+2}$ with a Lagrangian submanifold $\Lcal\subset\bC^{2n}$, and vice versa.

Fix a linear symplectic splitting $\bC^{2n+2} = \bC^2\times\bC^{2n}$ of the symplectic vector space $\bC^{2n+2}$ with its standard symplectic form $\hat\Omega$ and linear Darboux coordinates $z^0,w_0$ in $\bC^2$ such that the symplectic form on $\bC^2$ is given by $\diff z^0\wedge\diff w_0$. Then the symplectic vector space $\bC^{2n}$ with its standard symplectic form $\Omega$ is recovered as the symplectic reduction with respect to the Hamiltonian flow of the function $z^0$ as explained in Rem.~\ref{rem:symplecticreductionexplained}. Let $\pi:\{z^0=1\}\to\{z^0=1\}/\langle \partial_{w_0} \rangle = \bC^{2n}$ be the quotient map and $\iota : \{z^0=1\} \hookrightarrow \bC^{2n+2}$ the inclusion.

In one direction, let $\Lcal$ be a Lagrangian submanifold of $\bC^{2n}$. A submanifold $\hat\Lcal_1\subset\{z^0=1\} \subset\bC^{2n+2}$
is called a \emph{lift} of $\Lcal$ if the projection
\begin{equation}
  \pi\vert_{\hat\Lcal_1}: \hat\Lcal_1 \to \Lcal
\end{equation}
is a diffeomorphism. Equivalently, a lift is a section over $\Lcal$ of the trivial $\bC$-bundle $\pi:\{z^0=1\}\to\bC^{2n}$. Hence, a lift $\hat\Lcal_1$ is of the form $\hat\Lcal_1 = \{(1,f(q),q) \mid q\in\Lcal\}$ for a function $f:\Lcal \to \bC$.

\begin{Prop}
\label{prop:conificationislagrange}
  Let $\hat\Lcal_1$ be a lift of a Lagrangian submanifold $\Lcal\subset\bC^{2n}$ with respect to the function $f:\Lcal \to\bC$. Then the cone $\hat\Lcal := \bC^*\cdot \hat\Lcal_1$ is Lagrangian if and only if $f$ is a Lagrangian potential.
\end{Prop}

\begin{proof}
  By the above $\hat\Lcal_1 = \{(1,f(q),q) \mid q\in\Lcal\}$. To show that $\hat\Lcal:=\bC^*\cdot\hat\Lcal_1$ is Lagrangian it is sufficient to show that $\hat\Omega(p,\hat X_p)=0$ for all $p\in\hat\Lcal_1$ and $\hat X_p\in T_p\hat\Lcal_1$.
  A tangent vector $\hat X_p\in T_p\hat\Lcal_1$ is of the form $\hat X_p = \diff f(X)\partial_{w_0} + X$ for $X\in T_q\Lcal$ with $q=\pi(p)$. Then
  \begin{align}
    \hat\Omega(p, \hat X_p)
    &=
      \hat\Omega(\partial_{z^0}+f(q)\partial_{w_0} + q, \hat X_p) \\
    &=
      \diff z^0\wedge\diff w_0\left(\partial_{z^0}+f(q)\partial_{w_0}, \diff f(X)\partial_{w_0}\right)
      + \Omega(q, X) \\
    &=
      \diff f(X) + \eta_q(X).
  \end{align}
  Hence, $\hat\Lcal$ is Lagrangian if and only if $\diff f = -\eta\vert_\Lcal$.
\end{proof}

\begin{Def}
  \label{def:conification}
Let $\hat\Lcal_1$ be the lift of the Lagrangian pair $(\Lcal, f)$. We call the Lagrangian cone $\con(\Lcal, f) := \bC^*\cdot\hat\Lcal_1$ the \emph{conification} of $(\Lcal, f)$.
\end{Def}

Conversely, let $\hat\Lcal\subset\bC^{2n+2}$ be a Lagrangian cone such that the submanifold $\hat\Lcal_1:=\hat\Lcal\cap\{z^0=1\}$ is transverse to the Hamiltonian vector field $\partial_{w_0}$ and each integral curve intersects $\hat\Lcal_1$ at most once. We will call Lagrangian cones with this property \emph{regular}. Then we define $\Lcal\subset\bC^{2n}$ as the image of $\hat\Lcal_1$ under the quotient map $\pi:\{z^0=1\}\to\{z^0=1\}/\langle \partial_{w_0} \rangle = \bC^{2n}$. Since the pullback $\pi^*\Omega$ of the symplectic form $\Omega$ on $\bC^{2n}$ is given by $\pi^*\Omega = \iota^*\hat\Omega$, it follows that $\Lcal$ is Lagrangian. By the regularity, the function $f := w_0\circ\inv{(\pi\vert_{\hat\Lcal_1})}$ is a well-defined function on $\Lcal$ and $\hat\Lcal_1$ is of the form $\hat\Lcal_1 = \{(1,f(q), q) \mid q\in\Lcal\}$. In particular, $\hat\Lcal_1$ is the lift of $\Lcal$ with respect to the function $f$.

\begin{Def}
  \label{def:reduction}
  We call the pair $\red(\hat\Lcal):=(\Lcal,f)$ the \emph{reduction} of the regular Lagrangian cone $\hat\Lcal\subset\bC^{2n+2}$.
\end{Def}



\begin{Prop}
  With respect to a splitting $\bC^{2n+2}=\bC^2\times\bC^{2n}$ and linear Darboux coordinates $z^0,w_0$ of $\bC^2$, we obtain a one-to-one correspondence
  \begin{equation}
    \label{eq:onetooneconification}
    \xymatrixcolsep{3pc}
    \xymatrix{
      \{\text{Lagrangian pairs $(\Lcal, f)$ in $\bC^{2n}$}\}
      \ar@{<->}[r]^-{\text{1:1}}
      &
      \{\text{Regular Lagrangian cones in $\bC^{2n+2}$}\}
    }
  \end{equation}
  given by conification and reduction.

  Moreover, conification and reduction are equivariant with respect to the action of the group $\Gc$, i.e., $\con(x\cdot(\Lcal,f)) = \rho(x)\con(\Lcal,f)$ and $\red(\rho(x)\hat\Lcal) = x\cdot\red(\hat\Lcal)$ for $x\in \Gc$.
\end{Prop}

\begin{proof}
  Let $\hat\Lcal\subset\bC^{2n+2}$ be a regular Lagrangian cone. We have already seen that $\hat\Lcal_1 = \hat\Lcal \cap \{z^0 = 1\}$ is the same as the lift of the pair $(\Lcal, f) := \red(\hat\Lcal)$. Since the cone $\hat\Lcal = \bC^*\cdot\hat\Lcal_1$ is Lagrangian, it follows from Prop.~\ref{prop:conificationislagrange} that $f$ is a Lagrangian potential and, hence, $\con(\red(\hat\Lcal)) = \hat\Lcal$.
  Conversely, if $(\Lcal, f)$ is a Lagrangian pair and $\hat\Lcal_1\subset\{z^0=1\}$ is the lift of $\Lcal$ with respect to $f$, then $\con(\Lcal, f) = \bC^*\cdot\hat\Lcal_1$ is a regular Lagrangian cone by Prop.~\ref{prop:conificationislagrange}. Since $\con(\Lcal, f)\cap\{z^0=1\} = \hat\Lcal_1$, it follows that $\red(\con(\Lcal, f)) = (\Lcal, f)$. This shows $\red=\inv\con$.

  Now let $x=(X,s,v)\in \Gc$ and $\hat\Lcal_1$ be the lift of a Lagrangian pair $(\Lcal, f)$. Then
  \begin{align}
    \rho(x)\hat\Lcal_1
    &=
      \rho(x)\{ (1, f(q), q) \in\bC^{2n+2} \mid q\in\Lcal \}\\
    &=
      \{(1, f(q) + \hat v^tq - 2s, xq)\in\bC^{2n+2} \mid q\in\Lcal\} \\
    &=
      \{(1, f(q) + \Omega(xq, v) - 2s, xq) \in\bC^{2n+2}\mid q\in\Lcal \} \\
    &=
      \{(1, f(\inv x q') + \Omega(q', v) - 2s, q') \in\bC^{2n+2}\mid q'\in x\Lcal\} \\
    &=
      \{(1, x\cdot f(q'), q') \in\bC^{2n+2}\mid q'\in x\Lcal\}.
  \end{align}
  This shows that $\rho(x)\hat\Lcal_1$ is the lift of the Lagrangian pair $x\cdot(\Lcal,f) = (x\Lcal, x\cdot f)$. Since the action of $\Gc$ on $\bC^{2n+2}$ leaves level-sets of $z^0$ and the distribution spanned by $\partial_{w_0}$ invariant, it follows that
  \begin{equation}
    \con(x\cdot(\Lcal,f)) = \bC^*\cdot(\rho(x)\hat\Lcal_1) = \rho(x)(\bC^*\cdot\hat\Lcal_1) = \rho(x)\con(\Lcal,\hat w_0).
  \end{equation}
  The equivariance of $\red$ follows immediately from $\red = \inv\con$.
\end{proof}

\begin{Prop}
\label{prop:kaehlerianconificationcondition}
  Let $(\Lcal, f)$ be a Lagrangian pair such that $\Lcal$ is K\"ahlerian. If there is a point $q\in\Lcal$ such that $q$ is real and $f(q)\not\in\bR$, then there is an open neighborhood $U\subset\Lcal$ of $q$ such that the Lagrangian cone $\hat U := \con(U, f) \subset \hat\Lcal := \con(\Lcal, f)$ is K\"ahlerian.
\end{Prop}
\begin{proof}
  Let $q\in\Lcal$ be real such that $f(q)\not\in\bR$ and choose an arbitrary $\hat v \in T_p\hat\Lcal\cap\tau T_p\hat\Lcal$ for $p = (1, f(q), q)\in\hat\Lcal$. Since $T_p\hat\Lcal = \Span_\bC(p) \oplus T_q\Lcal$, we have $\hat v = \lambda(1, f(q), q) + (0, \diff f(v), v)$ for $\lambda\in\bC$ and $v\in T_q\Lcal$. The condition $\hat v - \tau\hat v = 0$ gives three equations
  \begin{align}
    0 &= \lambda - \widebar\lambda, \\
    0 &= \lambda f(q) - \widebar{\lambda f(q)} + \diff f(v) - \widebar{\diff f(v)}, \\
    0 &= \lambda q - \widebar{\lambda q} + v - \widebar v.
  \end{align}
  From the first, we immediately see that $\lambda \in \bR$.
  From the third we find $v - \widebar v = \lambda(\widebar q - q) =0$ since $q$ is a real point.
  But $v - \widebar v = 0$ is only possible if $v=0$ as $\Lcal$ is K\"ahlerian.
  The second equation then implies $\lambda(f(q)-\widebar{f(q)}) = 0$ which, as $f(q)\not\in\bR$, is only possible if $\lambda = 0$.
  Hence, $\hat v = 0$ and this shows $T_p\hat\Lcal\cap\tau T_p\hat\Lcal = 0$.
  Since $\hat\Lcal$ is Lagrangian, this is equivalent to the Hermitian form $\hat\gamma = \hat\Omega(\cdot, \tau\cdot)$ being non-degenerate when restricted to $\hat\Lcal$ at the point $p$.
  By continuity, it is then also non-degenerate on a neighborhood $\hat U_1\subset\hat\Lcal_1 = \hat\Lcal\cap\{z^0=1\}$ of $p$.
  Non-degeneracy is invariant under multiplication by $z^0\in\bC^*$, which acts by homothety on the Hermitian form $\hat\gamma$.
  Therefore, $\hat\gamma\vert_{\hat\Lcal}$ is non-degenerate on $\hat U := \bC^*\cdot\hat U_1$ which is the conification of the Lagrangian pair $(U, f)$ for $U = \pi(\hat U_1)$.
\end{proof}

\begin{Prop}
\label{prop:kaehlerianconificationexistence}
  If $(\Lcal, f)$ is a Lagrangian pair and $\Lcal$ is K\"ahlerian, then there is an open subset $U\subset\Lcal$ and an element $x\in \Gsk$ such that the cone $\con(x\cdot(U,f))$ is K\"ahlerian.
\end{Prop}
\begin{proof}
  Let $(\Lcal, f)$ be a Lagrangian pair such that $\Lcal$ is K\"ahlerian. If $\Lcal$ does not have real points, set $\Lcal' = \Lcal - q$ for an arbitray $q\in\Lcal$. Then $0\in\Lcal'$ is a real point and we can choose a Lagrangian potential $f'$ such that $f'(0) \not\in\bR$. This determines an element $x\in \Gsk$ such that $(\Lcal', f') = x\cdot(\Lcal, f)$. The statement now follows from Prop.~\ref{prop:kaehlerianconificationcondition}.
\end{proof}


\section{Conification of affine special K\"ahler manifolds}

\subsection{Conification of special K\"ahler pairs}

Since special K\"ahler pairs locally correspond to Lagrangian pairs we can use the results from the previous chapter to give a conification procedure for special K\"ahler pairs.

\begin{Prop}
Let $(\phi, F)$ be a special K\"ahler pair of an affine special K\"ahler manifold $M$ and denote by $(z,w):=\phi$ the components of $\phi$ as before.
Set $\hat M := \bC^*\times M = \{(z^0,p)\in\bC^*\times M\}$ with $\bC^*$-action defined by $\lambda\cdot(z^0,p) := (\lambda z^0, p)$. Then the map
  \begin{align}
  \Phi: \hat M
    &\to \bC^{2n+2} \\
    (z^0, p)
    &\mapsto
      z^0(1, (2F-z^tw)(p), \phi(p))
  \end{align}
  is a $\bC^*$-equivariant Lagrangian immersion of $\hat M$. 
\end{Prop}

\begin{proof}
  Consider open subsets $\hat U$ of $\hat M$ of the form $\hat U = \bC^*\times U$ where $U\subset M$ is open such that $\phi\vert_U$ is an embedding. Let $(\Lcal, f)$ be the Lagrangian pair corresponding to $(\phi, F)\vert_U$ by \Cref{lem:lagpotprepotcorrespondence}. Then $\Phi(z^0,p) = z^0(1, f(\phi(p)), \phi(p))$ for all $(z^0,p)\in\hat U$, i.e., $\Phi(\hat U) = \con(\Lcal, f)$. This shows that $\Phi$ is a Lagrangian immersion. The equivariance is obvious.
\end{proof}

\begin{Def}
  Let $(\phi, F)$ be a special K\"ahler pair of an affine special K\"ahler manifold $M$.
    We call the complex manifold $\hat M = \bC^*\times M$ together
    with the map $\Phi: \hat M\to\bC^{2n+2}$ the \emph{conification}
    of the special K\"ahler pair $(\phi, F)$ and we write
    $\Phi = \con(\phi, F)$.
    We say that the special K\"ahler pair $(\phi, F)$ is \emph{non-degenerate} if the immersion $\Phi$ is K\"ahlerian and $\hat\gamma(\Phi,\Phi) \neq 0$.
\end{Def}

\begin{Prop}\label{prop:Gskequivarianceoflocalconification}
  Let  $(\phi, F)$ be a special K\"ahler pair of an affine special K\"ahler manifold $M$. Then conification is equivariant with respect to the action of $\Gsk$ in the sense that $\con(x\cdot(\phi, F)) = \rho(x)\circ\con(\phi, F)$ for $x\in\Gsk$.
\end{Prop}
\begin{proof}
  This follows since conification locally corresponds to the conification of Lagrangian pairs.
\end{proof}

\begin{Th}\label{thm:localconification}
  Let $(\phi, F)$ be a non-degenerate special K\"ahler pair of an affine special K\"ahler manifold $M$. Then $\Phi=\con(\phi, F)$ induces on $\hat M$ the structure of a conical affine special K\"ahler manifold. This structure is independent of the representative of the equivalence class of $(\phi, F)$ in  $\Fcal(M)/G$.
\end{Th}

\begin{proof}
  Let $\Phi$ be the conification of a non-degenerate special K\"ahler pair $(\phi, F)$. Then $\Phi$ is by definition a K\"ahlerian Lagrangian immersion of $\hat M$ inducing the special K\"ahler metric $\hat g = \Re\Phi^*(\hat\gamma)$. Since $\Phi$ is also equivariant with respect to the $\bC^*$-action, it follows that the real part $\xi := \Re(Z)$ of the vector field $Z$ generating the $\bC^*$ action satisfies $\nabla\xi = D\xi = \Id$. Its length is given by
  \begin{align}\label{eq:conicallength}
    \hat g(\xi,\xi)
    = \frac12\hat\gamma(\Phi,\Phi) = {|Z^0|}^2(\Im f + K) \neq 0
  \end{align}
  where $f = 2F - z^tw$ for $(z,w):=\phi$ and $K = \frac12\gamma(\phi,\phi)$. This shows that $\Phi$ induces on $\hat M$ a conical affine special K\"ahler structure.

  Let $(\phi', F')\in \Fcal(M)$ with $\Phi' = \con(\phi', F')$. Then $(\phi', F') = x\cdot (\phi, F)$ for a unique $x\in\Gc$ and by \Cref{prop:Gskequivarianceoflocalconification} $\Phi' = \rho(x)\circ\Phi$. Now $\Phi$ and $\Phi'$ induce the same conical affine K\"ahler structure on $\hat M$ if and only if $\rho(x)\in\Sp(\bR^{2n+2})$ which is the case if and only if $x\in G$.
\end{proof}

\begin{Prop}\label{prop:nondegiffpsk} Let $(\phi, F)$ be a special K\"ahler pair defined on $U\subset M$ and set $f = 2F - z^tw$ for $(z,w) := \phi$ and $K = \frac12\gamma(\phi, \phi)$. Then $(\phi, F)$ is non-degenerate if and only if $\Im f + K \neq 0$ and $\widebar\omega:=\frac i2\partial\widebar\partial\log|\Im f + K|$ is non-degenerate.
\end{Prop}
\begin{proof}
  This follows easily from \cref{eq:conicallength,eq:pskmetric}
\end{proof}

\begin{rem}\label{rem:localconificationformulas}
  A special K\"ahler domain $M\subset\bC^n$ with coordinates $z^1,\ldots,z^n$ of $\bC^n$ and prepotential $F:M\to\bC$ determines a special K\"ahler pair $(\phi, F)$ by setting $\phi = \diff F: M\to T^*\bC^n=\bC^{2n}$. Then the conification
  \begin{equation}
    \label{eq:conificationofskdomain}
    \begin{aligned}
      \hat M &= \{(Z^0,Z^1,\ldots,Z^n)\in\bC^*\times\bC^{n} \mid Z^i/Z^0\in M, i=1,\ldots,n\}, \\
      \Phi &= \con(\diff F, F): \hat M \to \bC^{2n+2}
    \end{aligned}
  \end{equation}
  is the graph of $\diff\hat F$, where $\hat F$ is a holomorphic homogeneous function of degree 2 given by
    \begin{equation}
    \label{eq:conifiedprepotential}
    \hat F(Z^0,\ldots,Z^n) = {(Z^0)}^2 F\left( \frac{Z^1}{Z^0},\ldots,\frac{Z^n}{Z^0} \right).
  \end{equation}
  The special K\"ahler pair $(\phi, F)$ is non-degenerate if and only if the matrix given by $\Im\left(\frac{\partial^2\hat F}{\partial Z^I\partial Z^J}\right)$ for $I,J=0,\ldots,n$ is invertible and
  \begin{equation}
    \begin{aligned}
      \hat K(Z^0,\ldots,Z^n)
      &= \sum_{I=0}^n\Im\left(\widebar Z^I\frac{\partial F}{\partial Z^I}\right)\\
      &= \left| Z^0 \right|^2 \left( K(z^1,\ldots,z^n) + \Im(f(z^1,\ldots,z^n)) \right)
    \end{aligned}
  \end{equation}
  is non-zero, where $z^i = Z^i/Z^0$, $f = 2F - \sum_{i=1}^n z^i\frac{\partial F}{\partial z^i}$, and $K=\sum_{i=1}^n\Im(\widebar z^i\frac{\partial F}{\partial z^i})$. Note that in this case, $\hat K=\frac12\hat\gamma(\Phi,\Phi)$ is the K\"ahler potential, $\Im\left(\frac{\partial^2\hat F}{\partial Z^I\partial Z^J}\right)=\frac{\partial^2\hat K}{\partial Z^I\partial\widebar Z^J}$ are the components of the metric, and
  \begin{equation}
    K'(z^1,\ldots,z^n) := -\log|K(z^1,\ldots,z^n) + \Im(f(z^1,\ldots,z^n))| = -\log|\hat K(1,z^1,\ldots,z^n)|
  \end{equation}
  gives a K\"ahler potential of the projective special K\"ahler metric $\widebar g$ defined on $\hat M/\bC^*\cong M$.
\end{rem}

\begin{Ex}
  Let $M\subset\bC^n$ with standard coordinates $(z^1,\ldots,z^n)$ be an affine special K\"ahler domain with a holomorphic prepotential $F = \sum_{i,j=1}^n a_{ij}z^iz^j +\frac12C$ for $a_{ij},C\in\bC$. Note how the parameter $C$ does not affect the affine special K\"ahler geometry of $M$. We have $K = \sum_{i,j=1}^nz^i\widebar z^j\Im(a_{ij})$ and $f = 2F - \sum_{i=1}^n z^i\frac{\partial F}{\partial z^i} = C$. Consider the conification of the special K\"ahler pair $(\diff F, F)$. We denote by $(Z^0,\ldots,Z^n)$ the homogeneous coordinates on $\bC^*\times M$. The holomorphic prepotential $\hat F$ of the conification is then given by $\hat F(Z^0,Z) = \sum_{i,j=1}^n a_{ij}Z^iZ^j + C{(Z^0)}^2$.
  The matrix
  \begin{equation}
    \left(\Im\frac{\partial^2\hat F}{\partial Z^I\partial Z^J}\right)_{I,J=0,\ldots,n}
    =
    \begin{pmatrix}
      \Im C & 0 \\
      0 & (\Im a_{ij})_{i,j=1,\ldots,n}
    \end{pmatrix}
  \end{equation}
  is non-degenerate if and only if $c:=\Im C\neq 0$. Thus $(\diff F, F)$ is non-degenerate if and only if $c\neq0$ and $K +\Im f =  K +c \neq 0$ on $M$.

  Assuming $(\diff F, F)$ is non-degenerate, then the projective special K\"ahler metric $\widebar g$ on $M$ is given by
  \begin{align}
    \widebar g
    &= -\sum_{i,j=1}^n\frac{\partial^2}{\partial z^i\partial\widebar z^j}\log|K+c| \\
    &= -\frac1{K+c} g  + \frac{1}{(K+c)^2}(\partial K)(\widebar\partial K),
  \end{align}
  where $g$ is the affine special K\"ahler metric of $M$.
\end{Ex}

\subsection{The ASK/PSK-correspondence}
In this section we will give a global description of the conification procedure of the previous section and establish the ASK/PSK-correspondence which will assign a projective special K\"ahler manifold to any affine special K\"ahler manifold given a non-degenerate special K\"ahler pair. For this, we will prove that every affine special K\"ahler manifold admits a flat principal $\Gsk$-bundle. Using this bundle, we show that if the holonomy of the flat connection is contained in the group $G\subset\Gsk$, then the local conification of a non-degenerate special K\"ahler pair $(\phi, F)$ can be extended to the largest domain on which analytic continuation of $(\phi, F)$ is non-degenerate.



\begin{Lem}\label{lem:sheafflatbundle}
  Let $G$ be a Lie group and $\Fcal$ be a presheaf on a manifold $M$ with values in the category of principal homogeneous $G$-spaces. Then the disjoint union of stalks $P:=\dot\cup_{p\in M}\Fcal_p$ carries the structure of a principal $G$-bundle $\pi: P \to M$ with a flat connection $1$-form $\theta$ such that the horizontal sections of $P$ over $U$ are given by $\Fcal(U)$.
\end{Lem}
\begin{proof}
  Fix a point $p\in M$ and a neighborhood $U$ of $p$ such that $\Fcal(U)\neq\emptyset$. We claim that evaluation of sections, i.e., the map taking a section $s\in\Fcal(U)$ to its germ $[s]_p\in\Fcal_p$, is a bijection. Let $[s_V]_p\in\Fcal_p$, where $s_V\in\Fcal(V)$ for some open neighborhood $V$ of $p$. Without loss of generality, we can assume $V\subset U$. If $s\in F(U)$ is a section, then there is a unique $x\in G$ such that $x\cdot s\vert_V = s_V$. Hence, $x\cdot s$ and $s_V$ define the same germ at $p$. This shows the surjectivity. Now let $s,\tilde s=x\cdot s\in\Fcal(U)$ such that $[s]_p = [\tilde s]_p$. Then there is a neighborhood $V\subset U$ of $P$ such that $s\vert_V = \tilde s\vert_V$. Since $s = x\cdot \tilde s$ for a unique $x\in G$ this implies $x=e$, where $e\in G$ is the neutral element, showing the injectivity. It follows that the stalks of $\Fcal$ are also principal homogeneous $G$-spaces with $G$-action defined as $x\cdot[s]_p = [x\cdot s]_p$.

  Set $P=\dot\cup_{p\in M}\Fcal_p$ and $\pi: P\to M$, $[s]_p\mapsto p$. We can now consider a section $s\in\Fcal(U)$ as a section of $P$ over $U$ by setting $s(p):=[s]_p$. Choose an open covering $\mathcal U = (U_\alpha)_{\alpha\in I}$ such that $\Fcal(U_\alpha)\neq\emptyset$ and for each $U_\alpha$ pick a section $s_\alpha\in\Fcal(U_\alpha)$. Define $G$-equivariant maps $\Psi_\alpha: \inv\pi(U_\alpha) \to U_\alpha\times G$ such that $\Psi_\alpha(s_\alpha(p)) = (p,e)$. These maps are bijective by the first part of the proof. Let $U_{\alpha\beta} = U_\alpha\cap U_\beta$ be a non-empty overlap. Then $\Fcal(U_{\alpha\beta}) \neq \emptyset$
  and by the simply transitive action of $G$ on $\Fcal(U_{\alpha\beta})$ there is a unique $x_{\alpha\beta}\in G$ such that $s_\alpha = x_{\alpha\beta}s_\beta$, showing that the transition maps
  \begin{align}
    \Psi_{\alpha\beta}(p,g) := (\Psi_\beta\circ\inv{\Psi_\alpha})(p, g)
    = \Psi_\beta(g\cdot s_\alpha(p))
    = \Psi_\beta(gx_{\alpha\beta}\cdot s_\beta(p))
    = (p, gx_{\alpha\beta})
  \end{align}
    are smooth and the transition functions $g_{\alpha\beta}: U_{\alpha\beta}\to\Gsk$, $g_{\alpha\beta}(p) = x_{\alpha\beta}$ are constant. On a non-empty overlap $U_{\alpha\beta\gamma} = U_\alpha\cap U_\beta \cap U_\gamma$ we have $s_\beta = x_{\beta\gamma} \cdot s_\gamma$ and $s_\alpha = x_{\alpha\beta} \cdot s_\beta = x_{\alpha\beta}x_{\beta_\gamma}\cdot s_\gamma$. Hence, the transition functions satisfy $g_{\alpha\gamma} = g_{\alpha\beta}g_{\beta\gamma}$. This shows that $\pi: P\to M$ is a principal $\Gsk$ bundle, see, e.g., \cite{kobayashinomizu1}*{Chapter 1, Proposition 5.2}).

  The transformation rule for local connection 1-forms $\theta_\alpha\in\Omega^1(U_\alpha, \Lie(\Gsk))$ is
  \begin{equation}
    \theta_\beta = \Ad(\inv{g_{\alpha\beta}})\theta_\alpha + \inv{g_{\alpha\beta}}\diff g_{\alpha\beta}
  \end{equation}
  for transition functions $g_{\alpha\beta}: U_{\alpha\beta}\to G$. In our case, the transition functions $g_{\alpha\beta}(p) = x_{\alpha\beta}$ are constant. Thus we see that setting $\theta_\alpha = 0$ defines a flat connection 1-form $\theta$ on $P$.

  In the above we have seen that a section $s\in \Fcal(U)$ gives a local trivialization $\Psi: \inv\pi(U) \to U\times G$. A section $\tilde s$ of $\inv\pi(U)$ is horizontal with respect to $\theta$ if and only if it is constant in this trivialization. Thus it is of the form $\tilde s(p) = [x\cdot s]_p$ for some $x\in G$. Under the identification $\Fcal_p \cong \Fcal(U)$, $\tilde s$ thus corresponds to $x\cdot s \in \Fcal(U)$, completing the proof.
\end{proof}

Now let $(M, J, g, \nabla)$ be an affine special K\"ahler manifold of complex dimension $n$. Consider the map $\mathcal F$ assigning to each open subset $U$ of $M$ the set $\Fcal(U)$ of special K\"ahler pairs of $U$. The map $\mathcal F$ is a sheaf with values in the category of $\Gsk$-principal homogeneous spaces. The restriction map is given by $(\phi, F)\vert_V = (\phi\vert_V, F\vert_V)$. By \Cref{lem:sheafflatbundle} the sheaf $\Fcal$ thus defines a flat principal $\Gsk$-bundle $\pi: P\to M$ with flat connection 1-form $\theta$ where $P = \dot\cup_{p\in M}\Fcal_p$.

\begin{Def}
  We call the flat principal $\Gsk$-bundle of germs of special K\"ahler pairs  $\pi: P \to M$ the \emph{bundle of special K\"ahler pairs}.
\end{Def}

\begin{Def}
  \begin{enumerate}
  \item We call a germ $u$ in the fiber $P_p$ \emph{non-degenerate} if
    there is a non-degenerate special K\"ahler pair $(\phi, F)$ of an open neighborhood of $p$ such that $[(\phi, F)]_p = u$. Note that every fiber contains at least one non-degenerate germ by \Cref{prop:kaehlerianconificationexistence}.
  \item Let $u=[(\phi, F)]_p$ be a non-degenerate germ in the fiber $P_p$ and $(\phi, F)$ be a non-degenerate special K\"ahler pair. Define $\dom(u)\subset M$ to be the set of points in $M$ that are connected to $p$ via a path $\gamma$ along which the analytic continuation of $(\phi, F)$ is non-degenerate. We call $\dom(u)$ the \emph{domain of non-degeneracy of $u$}.
\end{enumerate}
\end{Def}
Note that analytic continuation of a special K\"ahler pair $(\phi, F)$ defined on a neighborhood of a point $p$ along a path $\gamma$ corresponds to parallel transport of the germ $u=[(\phi,F)]_p\in P_p$ along $\gamma$. Therefore, if $u$ is non-degenerate, then a point $p'\in M$ is in $\dom(u)$ if and only if there is a horizontal path from $u$ to the fiber over $p'$ such that all points of $\gamma$ are non-degenerate.

\begin{Th}
  \label{thm:conditionsglobalconification}
  Let $M$ be a connected affine special K\"ahler manifold of complex dimension $n$ and $\pi: P\to M$ be the bundle of special K\"ahler germs of $M$ with its flat connection 1-form $\theta$. Assume that $\Hol(\theta) \subset G$. Let $u\in P$ be a non-degenerate point. Then the manifold $\hat M_u:= \bC^*\times \dom(u)$ carries a conical affine special K\"ahler structure.
\end{Th}
\begin{proof}
  Due to the condition on the holonomy, we can reduce the bundle $\pi: P\to M$ and the connection 1-form $\theta$ to a $\Hol(\theta)$-bundle
  \begin{equation}
    P(u) := \{u' \in P \mid \text{there is a $\theta$-horizontal path connecting $u$ and $u'$}\} \subset P.
  \end{equation}
  First note that if $u'\in P(u)_{p'}$ is a non-degenerate germ in the fiber over $p'$, then all germs in the fiber are non-degenerate. Indeed, if $u''\in P(u)_{p'}$, then $u'' = x\cdot u'$ for some $x\in\Hol(\theta)\subset G$. Thus if $(\phi', F')$ is the non-degenerate special K\"ahler pair corresponding to $u'$ then $\con(x\cdot(\phi', F')) = \rho(x)\con(\phi', F')$ is K\"ahlerian since $\rho(x)\in\Sp(\bR^{2n})$ for all $x\in G$.

  By the definition of $\dom(u)$ the fibers of $P(u)\vert_{\dom(u)}$ are all non-degenerate.
  Hence, we can find an open covering $\mathcal U = (U_\alpha)_{\alpha\in I}$ of $\dom(u)$ and non-degenerate special K\"ahler pairs $(\phi_\alpha, F_\alpha)\in\Fcal(U_\alpha)$ such that $[(\phi_\alpha, F_\alpha)]_p \in P(u)_p$ for all $p\in \dom(u)$.
  This gives a covering $\hat{\mathcal U} = (\hat U_\alpha) := (\bC^*\times U_\alpha)_{\alpha\in I}$ and conic K\"ahlerian Lagrangian immersions $\Phi_\alpha =  \con(\phi_\alpha,F_\alpha): \hat U_\alpha\to\bC^{2n+2}$.
  The induced conical affine special K\"ahler structure on $\hat U_\alpha$ is independent of the choice of special K\"ahler pairs $(\phi_\alpha, F_\alpha)$ for each $\alpha\in I$ by \Cref{thm:localconification} and agrees on overlaps, since the transistion functions take values in $\Sp(\bR^{2n+2})$. This shows that the $\Phi_\alpha$ induce a well-defined conical affine special K\"ahler structure on $\hat M_u = \bC^*\times \dom(u)$.
\end{proof}


The $\bC^*$-action on $\hat M_u$ is principal. Hence, the quotient $\widebar M_u = \hat M_u/\bC^*$ is projective special K\"ahler with metric $\widebar g_u$ given by \cref{eq:pskmetric}. In particular, a K\"ahler potential of $\widebar g_u$ is given by $K'_u(p) := -\log |\hat K_u(1,p)|$ for $p\in\dom(u)$.

\begin{Def}\label{def:askpsk}
  We call the map taking the affine special K\"ahler manifold $(M,g)$ and a special K\"ahler germ $u$ of $M$ to the projective special K\"ahler manifold $(\widebar M_u, \widebar g_u)$ the \emph{ASK/PSK-correspondence}.
\end{Def}

\section{Completeness of Hessian metrics associated with a hyperbolic centro\-affine hypersurface}
\label{sec:elemdefs}
In this section we will prove a completeness result for a one-parameter deformation of a positive definite Hessian metric with Hesse potential of the form $-\log h$ where $h$ is a homogeneous function on a domain in $\bR^n$. The latter metric is isometric to a product of the form $\diff r^2 + g_\Hcal$, where $g_\Hcal$ is proportional to the canonical metric on
a centroaffine hypersurface $\Hcal\subset\bR^n$. This will be specialized in \cref{sec:r-map} to the case of a cubic polynomial $h$ and related to the r-map.

Let $U\subset\bR^{n}$ be a domain such that $\bR^{>0}\cdot U\subset U$ and let $h:U\to\bR$ be a smooth positive homogeneous function of degree $k>1$.
Then
$\Hcal:= \{h=1\}\subset U$ is a smooth hypersurface and $U = \bR^{>0}\cdot\Hcal$. We assume that for $g_U:=-\partial^2h$ the metric $g_\Hcal := \iota^*g_U$ is positive definite, where $\iota:\Hcal\hookrightarrow U$ is the inclusion.
The manifold $\left(\Hcal, \frac1kg_\Hcal\right)$ is a hyperbolic centroaffine hypersurface in the sense of \cite{cns16}.

\begin{Def}
  If $h$ is a cubic homogeneous polynomial, then the manifold $(\Hcal, g_\Hcal)$, defined as above, is called a \emph{projective special real manifold}.
\end{Def}

Let $\prim g := -\partial^2\log h = \frac1hg_U+\frac1{h^2}(\diff h)^2$.
Denote by $\xi:=x^i\partial_{x^i}$ the position vector field on $U$ and by $E\subset TU$ the distribution of tangent spaces tangent to the level sets of $h$. Then $TU$ decomposes into
\begin{equation}
  \label{eq:TUdecomp}
  TU = E\oplus
  \left<
    \xi
  \right>.
\end{equation}

\begin{Prop}
  \label{prop:formulasprimg}
  The bilinear form $\check g := g_U - \frac{g_U(\xi,\cdot)^2}{g_U(\xi,\xi)}$ is positive semidefinite with kernel $\bR\xi$, and we can write
  \begin{align+}
    g_U       &= \check g - \frac{k-1}{kh}(\diff h)^2,
                \label{eq:gu}\\
    \prim g &= \frac1h\check g + \frac1{kh^2}(\diff h)^2.
              \label{eq:gprim}
  \end{align+}
  In particular, $g_U$ is a Lorentzian metric, $\prim g$ is a Riemannian metric on $U$, and the decomposition (\ref{eq:TUdecomp}) is orthogonal with respect to $g_U$ and $\prim g$.

\end{Prop}

\begin{proof}
  By homogeneity of $h$, we have $\diff h(\xi) = kh$, $g_U(\xi,\cdot) = -(k-1)\diff h$ and $g_U(\xi,\xi) = -k(k-1)h$. This implies $\check g\vert_{E\times E} = g_U\vert_{E\times E}>0$ and, hence, $\ker \check g = \bR\xi$. Observing that
  $
  \frac{g_U(\xi,\cdot)^2}{g_U(\xi,\xi)} = -\frac{(k-1)}{kh}(dh)^2
  $
  we obtain the formulas for $g_U$ and $\prim g$. The distributions $E$ and $\bR\xi$ are obviously orthogonal with respect to $\check g$ and $(dh)^2$ and, therefore, also with respect to $g_U$ and $g'$ which are linear combinations (with functions as coefficients) of these two tensors.
  %
\end{proof}

\begin{Def}\label{def:gprimc}
  For $c\in\bR$ we define the bilinear symmetric form
  \begin{equation}
    \prim g_c:=-\partial^2\log(h+c) = \frac1{h+c}g_U + \frac1{(h+c)^2}(\diff h)^2
    \label{eq:defgprimc}
  \end{equation}
  on the set
  \begin{equation}
    \label{eq:Uc}
    U_c :=
    \begin{cases}
      \{x\in U \mid h(x) + c > 0\} & \text{for $c\leq0$,} \\
      \{x\in U \mid h(x) - c(k-1) > 0\} & \text{for $c>0$.}
    \end{cases}
  \end{equation}
\end{Def}

\begin{Prop}\label{prop:propertiesgprimc}
  \begin{enumerate}
  \item As in \Cref{prop:formulasprimg} we can write
    \label{item:formulagprimc}
    \begin{equation}
      \label{eq:formulagprimgc}
      \prim g_c = \frac1{h+c}\check g + \frac{h-c(k-1)}{kh}\frac1{(h+c)^2}(\diff h)^2.
    \end{equation}
  \item \label{item:gpriemannian} The metric $g'_c$ is Riemannian on $U_c$.
  \item \label{item:gpisometric} 
    If $cc' > 0$,
    then $(U_c,g'_c)$ is isometric to $(U_{c'},g'_{c'})$.
%
%
  \end{enumerate}
\end{Prop}
\begin{proof}
  \begin{enumerate}
  \item \Cref{eq:formulagprimgc} follows by inserting \eqref{eq:gu} into  \eqref{eq:defgprimc}.
    \item The positive definiteness of $g'_c$ follows directly from \cref{eq:formulagprimgc} since the coefficients of the two terms are positive.
    \item Scalar multiplication by $\lambda>0$ is a diffeomorphism on $U$. Let $\phi_\lambda:U_c\to U$ be the restriction.
      Using the homogeneity of $h$ it easily follows that $\phi_{\lambda}(U_c) = U_{\lambda^k c}$.

    Computing
    \begin{align}
      \phi_\lambda^*\prim g_c
      &= \phi_\lambda^*
        \left(
        \frac1{h+c}g_U + \frac1{(h+c)^2}(\diff h)^2
        \right) \\
      &= \frac1{\lambda^kh+c}\lambda^kg_U + \frac1{(\lambda^kh+c)^2}\lambda^{2k}(\diff h)^2 \\
      &= \frac1{h+\lambda^{-k}c}g_U + \frac1{(h+\lambda^{-k}c)^2}(\diff h)^2\\
      &= \prim g_{\lambda^{-k}c}
    \end{align}
    we see that for $\lambda=(\prim c/c)^{1/k}$ we have $\phi_\lambda^*(\prim g_{\prim c}) = \prim g_c$. Hence, $\phi_\lambda$ gives the required isometry.
    %
\qedhere
  \end{enumerate}
\end{proof}

\begin{Th}\label{prop:primgccomplete}
  Assume that $\prim g$ is a complete metric on $U$ and $c<0$. Then $\prim g_c$ is a complete metric on $U_c$.
\end{Th}
\begin{rem}
  The metric $\prim g$ on $U$ is complete if and only if $g_\Hcal$ is complete, since $(U,\prim g)$ is isometric to $(\bR\times\Hcal, \diff r^2 + g_\Hcal)$.
\end{rem}
\begin{proof}
  Denote by $L(\gamma)$ and $\prim L_c(\gamma)$ the Riemannian length of a curve $\gamma$ in $U_c$ with respect to $\prim g$ and $\prim g_c$, respectively.
  Note first that
  \begin{equation}
    \label{eq:gprimcestimate}
    \begin{split}
      \prim g_c - \prim g &= \left(\frac1{h+c}-\frac1h\right)\check
      g + \frac1{k} \bigg( \underbrace{\frac{h-c(k-1)}{h}}_{>1}
      \frac1{(h+c)^2} - \frac1{h^2} \bigg)
      (\diff h)^2 \\
      &\geq \frac1{k} \left( \frac1{(h+c)^2}-\frac1{h^2} \right)
      (\diff h)^2 \geq 0
    \end{split}
  \end{equation}
  on $\prim U$. Hence, $\prim L_c(\gamma)\geq L(\gamma)$ for any curve $\gamma$ in $U_c$.

  Now, for some $T>0$ let $\gamma:[0,T)\to U_c$ be a curve that is not contained in any compact set in $U_c$. If $\gamma$ already has infinite length with respect to $\prim g$ then it also has infinite length with respect to $\prim g_c$ by \cref{eq:gprimcestimate} and we are done.

  Assume that $L(\gamma)<\infty$. Since $\prim g$ is complete, there exists a compact set $K\subset U$ such that $\gamma\subset K$. 
  Then $\{\gamma(t)\}$ has a limit point $p\in U$ that is not in $U_c$ because otherwise $\overline{\{\gamma(t)\}}\subset U_c$ is a compact subset of $U_c$ containing $\gamma$ which is a contradiction. By continuity of $h$, this limit point lies in $\{h+c=0\}$. Hence, we can find a sequence $t_i\in [0,T)$, $t_i\to T$, such that $h(\gamma(t_i)) \to -c$.

  Using the estimate
  \begin{align}
    \prim g_c
    &= \frac1{h+c}\check g+\frac{h-c(k-1)}{kh}(\diff\log(h+c))^2 \\
    &\geq \frac1{k}(\diff\log(h+c))^2
  \end{align}
  we find
  \begin{align}
    \prim L_c(\gamma)
    &\geq
      \frac1{\sqrt k}\int_0^{t_i}
      \left|
      \frac{\partial}{\partial t}\log(h(\gamma(t))+c)
      \right|\diff t \\
    &\geq
      \frac1{\sqrt k}
      \left|
      \log(h(\gamma(t_i))+c) - \log(h(\gamma(0))+c)
      \right|
      \stackrel{t_i\to T}{\longrightarrow} \infty
  \end{align}
  Hence, any curve that is not contained in any compact set in $U_c$ has infinite length with respect to $\prim g_c$. This is equivalent to the completeness of $\prim g_c$.
\end{proof}

\begin{rem}
  In the case of $c>0$ the metric $\prim g_c$ is not complete. One can find a curve with limit point in $\{h-c(k-1)=0\}$ that has finite length.
\end{rem}

The following lemma will be used in the proof of \Cref{thm:rmap} in the next section.
\begin{Lem}\label{lem:tangentmetriccomplete}
  Let $(M_1^n,g_1)$ be a complete Riemannian manifold. Then the metric
  \begin{equation}
    \label{eq:tangentmetric}
    g:= 
    \begin{pmatrix}
      g_1 & 0 \\
      0 & g_1 \\
    \end{pmatrix}
  \end{equation}
  defined on the product $M=M_1\times\bR^n$ is complete.
\end{Lem}
\begin{proof}
  This is a special case of \cite{cortes2012completeness}*{Theorem 2}.
\end{proof}

\section{Application to the r-map}
\label{sec:r-map}
Let us first recall the definition of the supergravity r-map, following \cite{cortes2012completeness}.

Let $(\Hcal, g_\Hcal)$ be a projective special real manifold defined by a homogeneous cubic polynomial $h$ such that $\Hcal \subset \{h = 1\}$. Set $U := \bR^{>0}\cdot\Hcal$ and define $g_U := -\partial^2 h$.

Define $\widebar M = \bR^n + iU\subset\bC^n$ with coordinates $(z^i = y^i + \sqrt{-1}x^i)_{i=1,\ldots,n}\in \bR^n+iU$. We endow $\widebar M$ with a K\"ahler metric $\widebar g$ defined by the K\"ahler potential $K(z) = -\log h(x)$. As a matrix, this metric is given by
$
\widebar g
=
\frac14
\begin{pmatrix}
  -\partial^2\log h(x) & 0 \\
  0 & -\partial^2\log h(x)
\end{pmatrix}.
$
Take note that $\widebar g$ is positive definite and is the quotient metric of the conical affine special K\"ahler manifold $\bC^*\times\widebar M$ defined by the prepotential $\hat F(Z^0,\ldots,Z^n) = -h(Z^1,\ldots,Z^n)/Z^0$, where $Z^0$ is the coordinate in the $\bC^*$-factor and $Z^i:=Z^0z^i$ for $i=1,\ldots,n$.
\begin{Def}
  The correspondence $(\Hcal, g_\Hcal) \mapsto (\widebar M, \widebar g)$ is called the \emph{supergravity r-map}.
\end{Def}

Related to the projective special real manifold $(\Hcal, g_\Hcal)$ is the so-called conical affine special real manifold $(U, g_U)$. The rigid r-map assigns it to the affine special K\"ahler manifold $(M:=\widebar M, g)$ with metric $g$ induced by the holomorphic prepotential $F(z) = -h(z)$. As a matrix with respect to the real coordinates $(y^i,x^i)$, this metric is given by
$
g
=
\begin{pmatrix}
  -\partial^2 h(x) & 0 \\
  0 & -\partial^2 h(x)
\end{pmatrix}.
$

Let $U_c$ be defined as in \cref{eq:Uc} and set $M_c = \bR^n + iU_c\subset M$. Note that $M_0 = M$.

\begin{Th}\label{thm:rmap}
  Applying the ASK/PSK-correspondence to the special K\"ahler pair
  \begin{equation}\label{eq:Fkorrektur}
    (\phi_c, F_c) := (\diff F, F-2\sqrt{-1}c)
  \end{equation}
  defined on $M_c$ with $F(z) = -h(z)$ and $c\in\bR$ gives a projective special K\"ahler manifold $(\widebar M_c, \widebar g_c)$. If $c=0$ we recover the supergravity r-map metric $\widebar g = \widebar g_0$. For any pair $c, c' \in \bR$ such that
  $cc'>0$
  the obtained manifolds $(\widebar M_c, \widebar g_c)$ and $(\widebar M_{c'}, \widebar g_{c'})$ are isometric.
  Moreover, if $c<0$ and $(\Hcal, g_\Hcal)$ is complete, then $(\widebar M_c, \widebar g_c)$ is complete.
\end{Th}
\begin{proof}
  We will use \Cref{prop:nondegiffpsk} to show that $(\diff F, F-2\sqrt{-1}c)$ is a non-degenerate special K\"ahler pair on $M_c$.
  Set $f(z) = 2(F-2\sqrt{-1}c) - \sum_{i=1}^nz^i\frac{\partial F}{\partial z^i} = h(z) - 4\sqrt{-1}c$ and $K(z) = \sum_{i=1}^n\Im\left(\widebar z^i\frac{\partial F}{\partial z^i}\right)$.
  Using the identity
  \begin{equation}
    \Im h(z) = \sum_{i=1}^n \Im\left(\widebar z^i\frac{\partial h}{\partial z^i}\right) - 4 h(\Im z),
  \end{equation}
  we compute $\Im f(z) + K(z) = -4(h(\Im z) + c)$, which is nonzero on $M_c$. The function $K' := -\log|\Im f + K| = -\log(4|h(\Im z) + c|)$ defines a symmetric bilinear tensorfield
  $
  \widebar g_c = \sum_{i,j=1}^n \frac{\partial^2 K'}{\partial z^i\widebar\partial z^j}\diff z^i\diff\widebar z^j
  $
  which, as a matrix, is of the form
  \begin{equation}
    \label{eq:deformedlocalrmapmetric}
    \widebar g_c = \frac14
    \begin{pmatrix}
      -\partial^2\log(h(x)+c) & 0 \\
      0 & -\partial^2\log(h(x)+c)
    \end{pmatrix}
    =
    \frac14
    \begin{pmatrix}
      g'_c(x) & 0 \\
      0 & g'_c(x)
    \end{pmatrix}
  \end{equation}
  where $\partial^2$ is the real Hessian operator with respect to the real coordinates $x$ and $g'_c$ is the deformed metric of the previous section. Hence, we see that $\widebar g_c$ is positive definite by  \Cref{prop:propertiesgprimc}. This proves that $(\diff F, F - 2\sqrt{-1}c)$ is a non-degenerate special K\"ahler pair on $M_c$. In particular, $\widebar g_c$ is the projective special K\"ahler metric that is obtained via \cref{eq:pskmetric} from the conical affine special K\"ahler metric $\hat g$ on the cone $\bC^*\times M_c$ with structure induced by $\con\left( \diff F, F-2\sqrt{-1}c \right)$. The supergravity r-map metric is recovered for $c=0$. If $g_\Hcal$ is complete and $c<0$, then $\widebar g_c$ is complete by \Cref{prop:primgccomplete} and \Cref{lem:tangentmetriccomplete}. It was proven in \Cref{prop:propertiesgprimc}.\ref{item:gpisometric} that scalar multiplication on $U$ by $\lambda>0$ induces a family of isometries $\phi_\lambda:(U_c, g'_c)\to (U_{\lambda^3c},g'_{\lambda^3c})$. The differential defines a corresponding family of isometries $\diff\phi_\lambda: (\widebar M_c = TU_c,\widebar g_c) \to (\widebar M_{\lambda^3c} = TU_{\lambda^3c}, \widebar g_c)$.
\end{proof}
\begin{rem}
  \label{rem:eldef}
  The above proof shows that the family of K\"ahler manifolds $(\widebar M_c, \widebar g_c)$ with $\widebar g_c$ given by \cref{eq:deformedlocalrmapmetric} is still defined when the projective special real manifold is replaced by a general hyperbolic centroaffine hypersurface associated with a homogeneous function $h$. The statements about completeness and isometries relating members of the family $(\widebar M_c, \widebar g_c)$ remain true under the assumption that the centroaffine hypersurface is complete.
  However, the metrics $\widebar g_c$ are in general no longer projective special K\"ahler. In fact, the ASK/PSK-correspondence can not be applied, as the K\"ahler metric $g$ obtained by the generalized r-map is in general no longer affine special K\"ahler.
  However, it turns out that the metrics $g$ and $\widebar g_c$ are related by an elementary deformation, as defined in \cite{swann2014elementary}*{Definition 1}, with the symmetry replaced by the vector field $X=\grad K_c$ for the K\"ahler potential $K_c=-4(h(\Im z)+c)$ and $g_\alpha:= g(X,\cdot)^2 + g(JX,\cdot)^2$. Indeed, the metric $\widebar g_c$ is of the form
  \begin{align}
    \label{eq:elementarydef}
    \widebar g_c
    &= f_1g + f_2g_\alpha \\
    &= \frac 1{K_c}\,g + \frac 1{4K_c^2}\left( (dK_c)^2 + (dK_c\circ J)^2 \right),
  \end{align}
  for $f_1=\frac1{K_c}$ and $f_2=\frac1{4K_c^2}$.
\end{rem}

\begin{Ex}\label{ex:ex1}
  Consider the complete projective special real manifold
  \begin{equation}
    \Hcal = \{(x,y,z)\in\bR^3 \mid x(xy - z^2) = 1, x > 0\}
  \end{equation}
  and set $U = \bR^{>0}\cdot\Hcal$. Computing the scalar curvature of the metric $g'_c := -\partial^2 \log(h+c)$ for $h = x(xy-z^2)$ and $c\in\bR$, for example with Mathematica \cite{mathematica} using the RGTC package \cite{rgtc}, gives
  \begin{equation}
    \scal_{g'_c} = -
    \frac
    {3(h^2  - 11ch + 6c^2)}
    {4(h-2c)^2}.
  \end{equation}
  For $c=0$ we find that $\scal_{g'_c} = -\frac34$ is constant. For $c\neq0$ we can further substitute $u := h/c$ and find
  \begin{equation}
    \scal_{g'_c} = -
    \frac
    {3(u^2 - 11u + 6)}
    {4(u -2)^2}
  \end{equation}
  which is constant only on the level sets of $h$. This shows that the deformed metrics are in general not isometric to the undeformed metric. Since the manifold $(U_c,g'_c)$ is contained in $(\widebar M_c, \widebar g_c)$ as a totally geodesic submanifold, this shows that the deformed metrics are in general not isometric to the undeformed metric.
\end{Ex}

\begin{Ex}\label{ex:ex2}
  Consider the complete projective special real manifold
  \begin{equation}
    \Hcal = \{(x,y,z)\in\bR^3 \mid xyz = 1, x > 0, y>0\}
  \end{equation}
  and set $U = \bR^{>0}\cdot\Hcal$. Computing the scalar curvature of the metric $g'_c := -\partial^2 \log(h+c)$ for $h = xyz$ and $c\in\bR$, gives
  \begin{equation}
    \scal_{g'_c} =
    \frac
    {3c(4h^2 -3ch + 2c^2) }
    {2h(h-2c)^2}.
  \end{equation}
  For $c=0$ we find that $\scal_{g'_c} = 0$ is constant. For $c\neq0$ we can substitute $u := h/c$ and find
  \begin{equation}
    \scal_{g'_c} =
    \frac{3(4u^2 - u + 2)}
    {2u(u-2)^2}
  \end{equation}
  which is constant only on the level sets of $h$.
\end{Ex}

\section{Acknowledgements}
This work was partly supported by the German Science Foundation (DFG) under the Research Training Group 1670 
and the Collaborative Research Center (SFB) 676. The work of T.M. was partly supported by the STFC consolidated grant ST/G00062X/1. He
thanks the Department of Mathematics and the Centre for Mathematical Physics
of the University of Hamburg for support and hospitality during various stages of this work.


\begin{bibdiv}
\begin{biblist}

\bib{acd02}{article}{
      author={Alekseevsky, DV},
      author={Cort{\'e}s, Vincente},
      author={Devchand, Chandrashekar},
       title={Special complex manifolds},
        date={2002},
     journal={J. Geom. Phys.},
      volume={42},
      number={1},
       pages={85\ndash 105},
}

\bib{acdm13}{article}{
      author={Alekseevsky, Dmitri~V.},
      author={Cort{\'e}s, Vicente},
      author={Dyckmanns, Malte},
      author={Mohaupt, Thomas},
       title={Quaternionic {K{\"a}hler} metrics associated with special
  {K{\"a}hler} manifolds},
        date={2015},
     journal={J. Geom. Phys.},
      volume={92},
       pages={271\ndash 287},
}

\bib{acm12}{article}{
      author={Alekseevsky, Dmitri~V},
      author={Cort{\'e}s, Vicente},
      author={Mohaupt, Thomas},
       title={{Conification of K{\"a}hler and hyper-K{\"a}hler manifolds}},
        date={2013},
     journal={Commun. Math. Phys.},
      volume={324},
      number={2},
       pages={637\ndash 655},
}

\bib{Antoniadis:1995ct}{article}{
      author={Antoniadis, Ignatios},
      author={Ferrara, Sergio},
      author={Gava, E.},
      author={Narain, K.S.},
      author={Taylor, T.R.},
       title={Perturbative prepotential and monodromies in {N=2} heterotic
  superstring},
        date={1995},
     journal={Nucl. Phys. B},
      volume={447},
      number={1},
       pages={35\ndash 61},
}

\bib{Antoniadis:1995vz}{article}{
      author={Antoniadis, Ignatios},
      author={Ferrara, Sergio},
      author={Taylor, TR},
       title={{N=2} heterotic superstring and its dual theory in five
  dimensions},
        date={1996},
     journal={Nucl. Phys. B},
      volume={460},
      number={3},
       pages={489\ndash 505},
}

\bib{Aspinwall:1996mn}{article}{
      author={Aspinwall, Paul~S.},
       title={K3 surfaces and string duality},
        date={1996},
     journal={Preprint, {\tt arXiv:hep-th/9611137}},
}

\bib{Behrndt:1996jn}{article}{
      author={Behrndt, Klaus},
      author={Cardoso, Gabriel~Lopes},
      author={de~Wit, Bernard},
      author={Kallosh, Renata},
      author={L\"ust, Dieter},
      author={Mohaupt, Thomas},
       title={Classical and quantum {N=2} supersymmetric black holes},
        date={1997},
     journal={Nucl. Phys. B},
      volume={488},
      number={1-2},
       pages={236\ndash 260},
}

\bib{rgtc}{inproceedings}{
      author={Bonanos, S.},
       title={Capabilities of the {Mathematica} package ``{Riemannian Geometry
  and Tensor Calculus}''},
organization={World Scientific Publishing Company Incorporated},
        date={2003},
   booktitle={{Proceedings of the 10th Hellenic Relativity Conference on Recent
  Developments in Gravity: Kalithea/Chalkidiki, Greece, 30 May-2 June 2002}},
       pages={174},
}

\bib{Cadavid:1995bk}{article}{
      author={Cadavid, AC},
      author={Ceresole, Anna},
      author={D'Auria, Riccardo},
      author={Ferrara, Sergio},
       title={11-dimensional supergravity compactified on {Calabi-Yau}
  threefolds},
        date={1995},
     journal={Phys. Lett. B},
      volume={357},
      number={1-2},
       pages={76\ndash 80},
}

\bib{Ceresole:1995jg}{article}{
      author={Ceresole, Anna},
      author={D'Auria, Riccardo},
      author={Ferrara, Sergio},
      author={Van~Proeyen, A.},
       title={Duality transformations in supersymmetric yang-mills theories
  coupled to supergravity},
        date={1995},
     journal={Nucl. Phys. B},
      volume={444},
      number={1-2},
       pages={92\ndash 124},
}

\bib{cds16}{article}{
      author={Cort\'es, Vicente},
      author={Dyckmanns, Malte},
      author={Suhr, Stefan},
       title={Completeness of projective special {K\"ahler} and quaternionic
  {K\"ahler} manifolds},
        date={2016},
     journal={to appear (accepted 26-12-2016) in ``New perspectives in
  differential geometry: special metrics and quaternionic geometry'' in honour
  of Simon Salamon (Rome, 16-20 November 2015), {\tt arXiv:1607.07232
  [math.DG]}},
}

\bib{cortes2012completeness}{article}{
      author={Cort\'es, V.},
      author={Han, X.},
      author={Mohaupt, T.},
       title={Completeness in {Supergravity} {Constructions}},
        date={2012},
        ISSN={0010-3616, 1432-0916},
     journal={Commun. Math. Phys.},
      volume={311},
      number={1},
       pages={191\ndash 213},
         url={http://link.springer.com/article/10.1007/s00220-012-1443-x},
}

\bib{cns16}{article}{
      author={Cort\'es, V.},
      author={Nardmann, M.},
      author={Suhr, S.},
       title={Completeness of hyperbolic centroaffine hypersurfaces},
        date={2016},
        ISSN={1019-8385},
     journal={Comm. Anal. Geom.},
      volume={24},
      number={1},
       pages={59\ndash 92},
         url={http://dx.doi.org/10.4310/CAG.2016.v24.n1.a3},
}

\bib{Candelas:1990rm}{article}{
      author={Candelas, Philip},
      author={Xenia, C.},
      author={Green, Paul~S.},
      author={Parkes, Linda},
       title={A pair of {Calabi-Yau} manifolds as an exactly soluble
  superconformal theory},
        date={1991},
     journal={Nucl. Phys. B},
      volume={359},
      number={1},
       pages={21\ndash 74},
}

\bib{dewit}{article}{
      author={de~Wit, Bernard},
       title={{N=2} electric-magnetic duality in a chiral background},
        date={1996},
     journal={Nucl. Phys. B},
      volume={49},
      number={1-3},
       pages={191\ndash 200},
}

\bib{deWit:1995zg}{article}{
      author={de~Wit, Bernard},
      author={Kaplunovsky, Vadim},
      author={Louis, Jan},
      author={L{\"u}st, Dieter},
       title={Perturbative couplings of vector multiplets in {N=2} heterotic
  string vacua},
        date={1995},
     journal={Nucl. Phys. B},
      volume={451},
      number={1-2},
       pages={53\ndash 95},
}

\bib{dewit92}{article}{
      author={de~Wit, Bernard},
      author={Van~Proeyen, Antoine},
       title={Special geometry, cubic polynomials and homogeneous quaternionic
  spaces},
        date={1992},
     journal={Comm. Math. Phys.},
      volume={149},
      number={2},
       pages={307\ndash 333},
}

\bib{ferrara90}{article}{
      author={Ferrara, Sergio},
      author={Sabharwal, S.},
       title={Quaternionic manifolds for type {II} superstring vacua of
  {Calabi-Yau} spaces},
        date={1990},
     journal={Nucl. Phys. B},
      volume={332},
      number={2},
       pages={317\ndash 332},
}

\bib{Grisaru:1986kw}{article}{
      author={Grisaru, Marcus~T.},
      author={Van~de Ven, A.E.M.},
      author={Zanon, D.},
       title={Four-loop divergences for the {N=1} supersymmetric non-linear
  sigma-model in two dimensions},
        date={1986},
     journal={Nucl. Phys. B},
      volume={277},
       pages={409\ndash 428},
}

\bib{haydys2008hyperkahler}{article}{
      author={Haydys, Andriy},
       title={{HyperK\"ahler and quaternionic K\"ahler manifolds with
  $S^1$-symmetries}},
        date={2008},
     journal={J. Geom. Phys.},
      volume={58},
      number={3},
       pages={293\ndash 306},
}

\bib{Hosono:1993qy}{article}{
      author={Hosono, Shinobu},
      author={Klemm, Albrecht},
      author={Theisen, S.},
      author={Yau, Shing-Tung},
       title={Mirror symmetry, mirror map and applications to {Calabi-Yau}
  hypersurfaces},
        date={1995},
     journal={Commun. Math. Phys.},
      volume={167},
      number={2},
       pages={301\ndash 350},
}

\bib{Harvey:1995fq}{article}{
      author={Harvey, Jeffrey~A.},
      author={Moore, Gregory},
       title={Algebras, {BPS} states, and strings},
        date={1996},
     journal={Nucl. Phys. B},
      volume={463},
      number={2-3},
       pages={315\ndash 368},
}

\bib{kobayashinomizu1}{book}{
      author={Kobayashi, Shoshichi},
      author={Nomizu, Katsumi},
       title={Foundations of differential geometry},
   publisher={Wiley},
        date={1963},
      volume={1},
}

\bib{Louis:1996mt}{article}{
      author={Louis, Jan},
      author={Sonnenschein, Jacob},
      author={Theisen, Stefan},
      author={Yankielowicz, Shimon},
       title={Non-perturbative properties of heterotic string vacua
  compactified on k3$\times$ t2},
        date={1996},
     journal={Nucl. Phys. B},
      volume={480},
      number={1-2},
       pages={185\ndash 212},
}

\bib{swann2014elementary}{book}{
      author={Macia, Oscar},
      author={Swann, Andrew},
       title={Elementary deformations and the {hyperK{\"a}hler}-quaternionic
  {K{\"a}hler} correspondence},
      series={in ``{Real} and complex submanifolds'', Daejoen, Korea, August
  2014, Springer Proc. Math. Stat.},
   publisher={Springer, Tokyo},
        date={2014},
      volume={106},
         url={http://dx.doi.org/10.1007/978-4-431-55215-4_30},
}

\bib{Nemeschansky:1986yx}{article}{
      author={Nemeschansky, Dennis},
      author={Sen, Ashoke},
       title={Conformal invariance of supersymmetric $\sigma$-models on
  {Calabi-Yau} manifolds},
        date={1986},
     journal={Phys. Lett. B},
      volume={178},
      number={4},
       pages={365\ndash 369},
}

\bib{mathematica}{manual}{
      author={{Wolfram Research, Inc.}},
       title={Mathematica 10.0},
         url={https://www.wolfram.com},
}

\end{biblist}
\end{bibdiv}
\end{document}
